\documentclass[11pt]{article}
\usepackage{fullpage}
\usepackage{amsmath,amssymb,amsthm}
\usepackage[english]{babel}
\usepackage[left=3cm,right=3cm,top=3cm,bottom=4cm]{geometry}

\usepackage{tikz,tikz-cd}
\usetikzlibrary{positioning}
\usepackage{braket}
\usepackage{mathtools,stmaryrd}

\SetSymbolFont{stmry}{bold}{U}{stmry}{m}{n}
\usepackage{lmodern,bm,bbm,mathrsfs}
\usepackage{manfnt}

\usepackage[textsize=tiny]{todonotes}

\usepackage[backref=page, colorlinks=true, linkcolor=blue, citecolor=red, menucolor=green]{hyperref}
\theoremstyle{plain}
\newtheorem*{theorem}{Theorem}
\newtheorem*{proposition}{Proposition}
\newtheorem*{lemma}{Lemma}
\newtheorem*{corollary}{Corollary}

\theoremstyle{definition}
\newtheorem*{definition}{Definition}
\newtheorem*{example}{Example}
\newtheorem*{remark}{Remark}

\renewcommand{\hat}[1]{\widehat{#1}}
\newcommand{\cat}[1]{\mathsf{#1}}

\renewcommand{\vec}[1]{\bm{#1}}
\renewcommand{\tilde}[1]{\widetilde{#1}}
\renewcommand{\bar}[1]{\overline{#1}}
\newcommand{\bC}{\mathbb{C}}
\newcommand{\bE}{\mathbb{E}}
\newcommand{\bF}{\mathbb{F}}
\newcommand{\bk}{\mathbbm{k}}
\newcommand{\bL}{\mathbb{L}}
\newcommand{\bP}{\mathbb{P}}
\newcommand{\bQ}{\mathbb{Q}}
\newcommand{\bR}{\mathbb{R}}
\newcommand{\bT}{\mathbb{T}}
\newcommand{\bZ}{\mathbb{Z}}
\newcommand{\cE}{\mathcal{E}}
\newcommand{\cF}{\mathcal{F}}

\newcommand{\cO}{\mathcal{O}}
\newcommand{\cV}{\mathcal{V}}
\newcommand{\cW}{\mathcal{W}}
\newcommand{\fB}{\mathfrak{B}}
\newcommand{\fM}{\mathfrak{M}}
\newcommand{\fN}{\mathfrak{N}}

\newcommand{\fX}{\mathfrak{X}}
\newcommand{\fY}{\mathfrak{Y}}
\newcommand{\scA}{\mathscr{A}}
\newcommand{\se}{\mathsf{e}}
\newcommand{\sT}{\mathsf{T}}
\newcommand{\sZ}{\mathsf{Z}}
\newcommand{\VW}{\mathsf{VW}}
\newcommand{\der}{\mathrm{der}}
\newcommand{\ind}{\mathsf{ind}}
\newcommand{\loc}{\mathrm{loc}}
\newcommand{\pt}{\mathrm{pt}}
\newcommand{\rig}{\mathsf{rig}}
\newcommand{\sst}{\mathrm{sst}}
\newcommand{\st}{\mathrm{st}}
\newcommand{\unrig}{\mathsf{unrig}}
\newcommand{\vir}{\mathrm{vir}}

\DeclareMathOperator{\ad}{ad}
\DeclareMathOperator{\cone}{cone}
\DeclareMathOperator{\cocone}{cocone}

\DeclareMathOperator{\GL}{GL}
\DeclareMathOperator{\Hom}{Hom}
\DeclareMathOperator{\cHom}{\mathcal{H}{\it om}}

\DeclareMathOperator{\PGL}{PGL}
\DeclareMathOperator{\Pic}{Pic}
\DeclareMathOperator{\rank}{rank}
\DeclareMathOperator{\Res}{Res}

\DeclareMathOperator{\tr}{tr}
\DeclareMathOperator{\tot}{tot}

\tikzset{%
  vertex/.style={shape=circle,fill=black,minimum size=6pt,inner sep=0},
  framing/.style={shape=rectangle,fill=black,minimum size=6pt,inner sep=0},
  baseline={([yshift=-0.8ex]current bounding box.center)}
}

\title{Semistable refined Vafa--Witten invariants}
\author{Henry Liu}
\date{\today}

\begin{document}

\maketitle

\begin{abstract}
  For any smooth complex projective surface $S$, we construct
  semistable refined Vafa--Witten invariants of $S$ which prove the
  main conjecture of \cite{Thomas2020}. This is done by extending part
  of Joyce's universal wall-crossing formalism to equivariant
  K-theory, and to moduli stacks with symmetric obstruction theories,
  particularly moduli stacks of sheaves on Calabi--Yau threefolds. An
  important technical tool which we introduce is the symmetrized
  pullback, along smooth morphisms, of symmetric obstruction theories.
\end{abstract}

\tableofcontents

\section{Introduction}

\subsection{}

Fix a smooth complex polarized surface $(S, \cO_S(1))$ and let
$\pi\colon \tot(K_S) \to S$ be the total space of its canonical
bundle. Let $\fM$ be the moduli stack \footnote{For this paper, all
  moduli stacks are already $\bC^\times$-rigidified unless indicated
  otherwise; see Remark~\ref{rem:rigidification}. In particular,
  stabilizers of stable points are {\it trivial}, not the group
  $\bC^\times$ of scaling automorphisms.} of compactly supported
torsion sheaves $E$ on the polarized Calabi--Yau threefold
\[ (X, \cO_X(1)) \coloneqq \left(\tot(K_S), \pi^*\cO_S(1)\right). \]
Let $\sT \coloneqq \bC^\times$, with coordinate denoted $t$, act on
$X$ and therefore also on $\fM$ by scaling the fibers of $\pi$ with
weight $t^{-1}$. For example, $K_X \cong t \otimes \cO_X$.

By the spectral construction \cite[\S 2]{Tanaka2020}, a point
$[E] \in \fM$ is equivalent to a {\it Higgs pair} $(\bar E, \phi)$
where $\bar E = \pi_*E \in \cat{Coh}(S)$ and
$\phi \in \Hom_S(\bar E, \bar E \otimes K_S)$. If $\bar\fM$ is the
moduli stack of all coherent sheaves on $S$, then this identifies
\[ \fM = T^*[-1]\bar\fM \]
as a $(-1)$-shifted cotangent bundle (of appropriate derived
enhancements) where $\sT$ scales the cotangent fibers. This
perspective is useful psychologically but we will avoid any actual use
of derived algebraic geometry. \footnote{The two exceptions are
  Example~\ref{ex:symmetrized-pullback-shifted-cotangent-bundle} and
  its use in the construction of
  \S\ref{sec:quiver-stacks-obstruction-theory}, but we also give an
  alternate non-derived construction in
  \S\ref{sec:quiver-stacks-obstruction-theory-alternate}.}

\subsection{}

Unsurprisingly, there is an exact triangle
\begin{equation} \label{eq:VW-obstruction-theory}
  R\Hom_X(E, E) \to R\Hom_S(\bar E, \bar E) \xrightarrow{\circ\phi - \phi\circ} t^{-1} \otimes R\Hom_S(\bar E, \bar E \otimes K_S) \xrightarrow{[1]}
\end{equation}
relating the obstruction theories of $E$ and $(\bar E, \phi)$. However
if $H^1(\cO_S)$ or $H^2(\cO_S)$ is non-zero, then the obstruction
theories of $\det \bar E \in \Pic(S)$ and $\tr \phi \in H^0(K_S)$ are
constant pieces of the second and third terms in
\eqref{eq:VW-obstruction-theory} respectively, and make invariants
vanish; see \S\ref{sec:U-pairs-invariants}. They are removed by taking
\begin{equation} \label{eq:VW-stack}
  \fN \coloneqq \{\det \bar E = L, \; \tr \phi = 0\} \subset \fM
\end{equation}
for a given $L \in \Pic(S)$. The induced obstruction theory for $\fM$
relative to $H^0(K_S) \times \Pic(S)$, and therefore for $\fN$ by
restriction, is given by the first term in the exact triangle
\begin{equation} \label{eq:VW-reduced-obstruction-theory}
  R\Hom_X(E, E)_\perp \to R\Hom_S(\bar E, \bar E)_0 \xrightarrow{\circ\phi - \phi\circ} t^{-1} \otimes R\Hom_S(\bar E, \bar E \otimes K_S)_0 \xrightarrow{[1]}
\end{equation}
where a subscript $0$ denotes traceless part. Note that it remains
symmetric, in the sense that
$R\Hom_X(E, E)_\perp^\vee \simeq t \otimes R\Hom_X(E, E)_\perp[3]$.

For details, particularly on the important step of how to construct
this obstruction theory in families, see \cite[\S 5]{Tanaka2020}.
\footnote{The constructions of \cite{Tanaka2020} are given on stable
  loci only, but Atiyah classes satisfying all the usual properties
  can be defined for general Artin stacks \cite{Kuhn2024}. So we will
  not worry about the details of how to construct the obstruction
  theory on the whole stack.}

\subsection{}

Consider Gieseker stability on $\fN$ with respect to $\cO_X(1)$. Let
$\fN_{r,L,c_2}$ be the component consisting of sheaves $\bar E$ with
$\det \bar E = L$ and rank and Chern classes
\begin{equation} \label{eq:classes}
  \alpha \coloneqq (r, c_1(L), c_2) \in \bZ_{>0} \oplus H^2(S) \oplus H^4(S).
\end{equation}
When $\alpha$ is such that there are no strictly semistable objects,
the semistable locus $\fN_{r,L,c_2}^\sst \subset \fN_{r,L,c_2}$ is a
quasi-projective scheme and the symmetric obstruction theory on
$\fN_{r,L,c_2}$ becomes perfect upon restriction. The fixed locus
$(\fN_{r,L,c_2}^\sst)^\sT$ is proper and therefore the {\it refined
  Vafa--Witten invariant}
\begin{equation} \label{eq:stable-VW-invariant}
  \VW_\alpha(t) \coloneqq \chi\left(\fN_{r,L,c_2}^\sst, \hat\cO^\vir\right) \in \bQ(t^{\frac{1}{2}})
\end{equation}
is well-defined by $\sT$-equivariant localization; see
\S\ref{sec:joyce-song} for details. Deformation invariance implies
that $\VW_\alpha(t)$ depends only on $c_1(L)$ instead of $L$. It is
convenient to write $\fN_\alpha \coloneqq \fN_{r,L,c_2}$ for some
(unspecified) choice of $L \in \Pic(S)$ with the first Chern class
specified by $\alpha$.

\subsection{}

When $\alpha$ has strictly semistable objects, it becomes more
difficult to define enumerative invariants. Following the philosophy
for generalized Donaldson--Thomas invariants \cite{Joyce2012}, Tanaka
and Thomas \cite[\S 6]{Tanaka2017} \cite[\S 5]{Thomas2020} proposed to
rigidify $\fN_\alpha^\sst$ using an auxiliary moduli stack
\[ \fN^{Q(k)}_{\alpha,1} = \left\{(E, s) : [E] \in \fN_\alpha, \; s \in H^0(E \otimes \cO_X(k))\right\} \]
of {\it Joyce--Song pairs}, for $k \gg 0$ such that $H^{>0}(E(k)) =
0$. There are no strictly semistable Joyce--Song pairs. The semistable
locus $\fN^{Q(k),\sst}_{\alpha,1}$ is a quasi-projective scheme with
proper $\sT$-fixed locus, and it carries a symmetric perfect
obstruction theory. (These facts are discussed in
\S\ref{sec:symmetrized-pullback}--\S\ref{sec:joyce-song}; see
Remark~\ref{rem:our-vs-JS-pairs} in particular.) The {\it refined
  pairs invariant}
\[ \tilde\VW_\alpha(k, t) \coloneqq \chi\left(\fN^{Q(k),\sst}_{\alpha,1}, \hat\cO^\vir\right) \in \bQ(t^{\frac{1}{2}}) \]
is well-defined by $\sT$-equivariant localization; see
\S\ref{sec:joyce-song} for details. The conjecture, which we prove in
this paper, was as follows. Define the symmetrized quantum integer
\begin{equation} \label{eq:quantum-integer}
  [n]_t \coloneqq (-1)^{n-1} \frac{t^{\frac{n}{2}} - t^{-\frac{n}{2}}}{t^{\frac{1}{2}} - t^{-\frac{1}{2}}},
\end{equation}
and let $\tau$ be the Gieseker stability condition (normalized Hilbert
polynomial) of Definition~\ref{def:stability-conditions}. The extra
sign $(-1)^{n-1}$ in $[n]_t$ saves on signs elsewhere.

\begin{theorem}[{\cite[Conjecture 5.2]{Thomas2020}}] \label{thm:VW-invars}
  There exist $\VW_{\alpha_i}(t) \in \bQ(t^{\frac{1}{2}})$ such that:
  \begin{enumerate}
  \item if $H^1(\cO_S) = H^2(\cO_S) = 0$, then
    \begin{equation} \label{eq:semistable-VW-invariant}
      \tilde\VW_\alpha(k, t) = \sum_{\substack{n > 0\\\alpha_1 + \cdots + \alpha_n = \alpha\\\forall i: \tau(\alpha_i) = \tau(\alpha)}} \frac{1}{n!} \prod_{i=1}^n \Big[\chi(\alpha_i(k)) - \chi\Big(\sum_{j=1}^{i-1} \alpha_j, \alpha_i\Big)\Big]_t \VW_{\alpha_i}(t);
    \end{equation}
  \item if $H^2(\cO_S) \neq 0$, then
    \begin{equation} \label{eq:semistable-VW-invariant-h2-nonzero}
      \tilde\VW_\alpha(k, t) = [\chi(\alpha(k))]_t \VW_\alpha(t);
    \end{equation}
  \item if $H^1(\cO_S) \neq 0$, then
    \eqref{eq:semistable-VW-invariant-h2-nonzero} holds modulo $(1 -
    t)^{\dim H^1(\cO_S)}$. \footnote{The invariants
    $\tilde\VW_\alpha(k, t)$ have no poles at $t=1$ \cite[Proposition
      2.22]{Thomas2020}, so it makes sense to filter by the order of
    vanishing at $t=1$. If additionally $H^2(\cO_S) \neq 0$, then
    \eqref{eq:semistable-VW-invariant-h2-nonzero} holds without any
    caveats.}
  \end{enumerate}
\end{theorem}

These $\VW_\alpha(t)$ are {\it refined semistable Vafa--Witten
  invariants}. In principle, there are many sensible ways to define
semistable invariants and this theorem is only one possible choice. We
expect that it is a good choice because it will have good
wall-crossing behavior under variation of stability condition, like in
\cite[Theorem 5.18]{Joyce2012}. This will be explored in future work.

The original statement \cite[Conjecture 5.2]{Thomas2020} assumed that
$\cO_S(1)$ is generic for $\alpha$, in the sense of
\eqref{eq:generic-polarization}, which makes $\chi(\alpha_j, \alpha_i)
= 0$ and simplifies the formula.

In the $H^1(\cO_S) \neq 0$ case, Thomas also conjectures that the
caveat ``modulo $(1 - t)^{\dim H^1(\cO_S)}$'' is unnecessary, but we
are unsure whether this is true.

\subsection{}

Note that Theorem~\ref{thm:VW-invars} becomes trivial if the
$\VW_\alpha(k, t)$ were allowed to depend on $k$, because
\eqref{eq:semistable-VW-invariant} would be an upper-triangular and
invertible transformation between $\{\tilde\VW_\alpha(k, t)\}_\alpha$
and $\{\VW_\alpha(k, t)\}_\alpha$ in $\bQ(t^{\frac{1}{2}})$. The
non-trivial content is therefore that the $\VW_\alpha(k, t)$ defined
(uniquely) in this way are actually independent of $k$. This is a
wall-crossing result which we prove using ideas from Joyce's recent
universal wall-crossing formalism \cite[Theorem 5.7]{Joyce2021}.
Specifically, we:
\begin{enumerate}
\item (\S\ref{sec:quiver-framed-stacks},
  \S\ref{sec:master-space-calculation}) construct a master space
  $M_\alpha \coloneqq \fN^{\tilde Q(k_1,k_2),\sst}_{\alpha,\vec 1}$
  with a $\bC^\times$-action whose fixed loci consist of
  $\fN_{\alpha,1}^{Q(k_1),\sst}$, $\fN_{\alpha,1}^{Q(k_2),\sst}$, and
  some ``interaction'' terms;
\item (\S\ref{sec:symmetrized-pullback},
  \S\ref{sec:quiver-framed-stacks}) put a symmetric perfect
  obstruction theory on $M_\alpha$ using {\it symmetrized pullback} of
  the symmetric obstruction theory on $\fN_\alpha$;
\item (\S\ref{sec:master-space-calculation}) take K-theoretic residue
  of the equivariant localization formula on $M_\alpha$ to get a
  wall-crossing formula relating $\tilde\VW_\alpha(k_1, t)$ and
  $\tilde\VW_\alpha(k_2, t)$;
\item (\S\ref{sec:semistable-invariants}) prove a combinatorial lemma
  which, applied to the wall-crossing formula, shows that
  $\{\VW_\alpha(k, t)\}_\alpha$ are independent of $k$.
\end{enumerate}
Various forms of master spaces and such localization arguments have
existed since the beginnings of modern enumerative geometry
\cite{Thaddeus1996, Mochizuki2009, Nakajima2011, Kiem2013}. Step 1 was
done in \cite{Joyce2021} and the general principle of step 3 has also
appeared previously (mostly in cohomology), but it appears that step 2
and 4 are genuinely new. Together, steps 1--3 demonstrate how master
space arguments work in equivariant K-theory using symmetric
obstruction theories, and should be very generally applicable.
\footnote{For instance, we demonstrate in \cite{KLT} how they can be
  applied, along with wall-crossing ideas from \cite[Theorem
    5.8]{Joyce2021}, to study wall-crossing for invariants like
  $\VW_\alpha(t)$ under variation of stability condition.}

\subsection{}

The formula \eqref{eq:semistable-VW-invariant} is a {\it quantized}
version of the Joyce--Song formula \cite[Equation 5.17]{Joyce2012}
expressing generalized Donaldson--Thomas invariants in terms of pair
invariants, in the sense that it becomes the original formula in the
$t \to 1$ limit. We expect that many other formulas in
\cite{Joyce2012} can be quantized using similar techniques, to be
explored in future work.

\subsection{}

In \S\ref{sec:symmetrized-pullback}, we review obstruction theories on
Artin stacks and define and construct symmetrized pullbacks. Very
generally, if $f\colon \fX \to \fY$ is a morphism of Artin stacks and
$f$ has a relative obstruction theory, then there is a notion of
compatibility between obstruction theories (resp. symmetric
obstruction theories) on $\fX$ and $\fY$ due to Manolache
\cite{Manolache2012} (resp. Park \cite[Theorem 0.1]{Park2021}).
Symmetrized pullback is an operation
(\S\ref{sec:symmetrized-pullback-construction}) which takes a
symmetric obstruction theory on $\fY$ and produces a compatible one on
$\fX$, under some fairly restrictive assumptions on the relative
obstruction theory for $f$. These assumptions are satisfied in our
Vafa--Witten setting basically because all the stacks involved are
secretly shifted cotangent bundles.

Symmetric pullback is crucial in step 2 above for equipping the master
space $M_\alpha$ with {\it any} perfect obstruction theory, let alone
a symmetric one; that the resulting perfect obstruction theory is
indeed still symmetric is a bonus feature that dramatically simplifies
step 3.

\subsection{Acknowledgements}

A good portion of this project is either inspired by or directly based
on Dominic Joyce's work \cite{Joyce2021}, and benefitted greatly from
discussions with him. I am also grateful for discussions and ongoing
collaborations with Arkadij Bojko, Nikolas Kuhn and Felix Thimm,
particularly regarding symmetrized pullback and obstruction theories.

This work was supported by the Simons Collaboration on Special
Holonomy in Geometry, Analysis and Physics.

\section{Symmetrized pullback and K-theory}
\label{sec:symmetrized-pullback}

\subsection{}

\begin{definition}
  Let $\fX$ be an Artin stack over a base Artin stack $\fB$, with
  action of a torus $\sT$, and let $D_{\cat{QCoh},\sT}^-(\fX)$ be its
  derived category of bounded-above $\sT$-equivariant complexes with
  quasi-coherent cohomology. The cotangent complex $\bL_{\fX/\fB} \in
  D_{\cat{QCoh},\sT}^-(\fX)$ has a canonical $\sT$-equivariant
  structure. Let $\bT_{\fX/\fB} \coloneqq \bL_{\fX/\fB}^\vee$ denote
  the tangent complex.

  An ($\sT$-equivariant) {\it obstruction theory} of $\fX$ over $\fB$
  is an object $\bE \in D_{\cat{QCoh},\sT}^-(\fX)$, together with a
  morphism $\phi\colon \bE\to \bL_{\fX/\fB}$ in
  $D_{\cat{QCoh},\sT}^-(\fX)$, such that $h^{\geq 0}(\phi)$ are
  isomorphisms and $h^{-1}(\phi)$ is surjective. An obstruction theory
  is:
  \begin{itemize}
  \item {\it perfect} if $\bE$ is perfect of amplitude contained in
    $[-1,1]$;
  \item {\it $n$-symmetric of weight $t$}, for $n \in \bZ$ and a
    $\sT$-weight $t$, if there is an isomorphism $\Theta\colon
    \bE\xrightarrow{\sim} t \otimes \bE^\vee[n-2]$ such that
    $\Theta^\vee[n-2]=\Theta$.
  \end{itemize}
  From here on, unless otherwise indicated: all objects and morphisms
  are $\sT$-equivariant, and we take $n=3$ and just say ``symmetric''.
\end{definition}

\subsection{}

Suppose $\fX$ is equipped with a symmetric obstruction theory $\bE$
which is perfect as a complex. Symmetry implies
\[ h^{-i}(\bE) = t \otimes h^{i-1}(\bE)^\vee = 0 \quad \forall i>2. \]
If in addition $h^1(\bL_\fX) = 0$, then $h^1(\bE) = h^1(\bL_\fX) = 0$
and $\bE$ would be a perfect obstruction theory. This will not hold in
general because objects in $\fX$ can have non-trivial automorphisms.

However, if $\fX$ has a stability condition with no strictly
semistable objects, then the stable locus $\fX^\st \subset \fX$ is
open and $\sT$-invariant. Stable objects have no automorphisms, so
$h^1(\bL_\fX)\big|_{\fX^\st} = h^1(\bL_{\fX^\st}) = 0$, and hence
$\bE\big|_{\fX^\st}$ is a symmetric perfect obstruction theory on
$\fX^\st$.

\subsection{}

\begin{definition}
  Let $f\colon \fX \to \fY$ be a morphism of Artin stacks with a
  relative obstruction theory $\phi_f\colon \bE_f \to \bL_f$. Two
  symmetric obstruction theories $\phi_\fX\colon \bE_\fX \to \bL_\fX$
  and $\phi_\fY\colon \bE_\fY \to \bL_\fY$, for the same weight $t$,
  are {\it compatible under $f$} if there are morphisms of exact
  triangles
  \begin{equation} \label{eq:symmetric-compatibility-diagram}
    \begin{tikzcd}
      \bE_f[-1] \ar{r} \ar[equals]{d} & t \otimes \bF^\vee[1] \ar{r}{\zeta^\vee[1]} \ar{d}{\eta^\vee[1]} & \bE_\fX \ar{d}{\zeta} \ar{r}{+1} & {} \\
      \bE_f[-1] \ar{r}{\delta} \ar{d}{\phi_f[-1]} & f^*\bE_\fY \ar{r}{\eta} \ar{d}{f^*\phi_\fY} & \bF \ar{d} \ar{r}{+1} & {} \\
      \bL_f[-1] \ar{r} & f^*\bL_\fY \ar{r} & \bL_\fX \ar{r}{+1} & {}
    \end{tikzcd}
  \end{equation}
  for some $\bF \in D^-_{\cat{QCoh},\sT}(\fX)$ such that the
  right-most column is $\phi_\fX$. Note that the symmetry of $\bE_\fY$
  here actually implies the symmetry of $\bE_\fX$: the upper-right
  square is invariant under $t \otimes (-)^\vee[1]$, except for
  possibly the upper right corner $\bE_{\fX}$, and so uniqueness of
  cones yields an isomorphism $\bE_\fX \cong t \otimes
  \bE_\fX^\vee[1]$.

  From a different perspective, if only $\phi_f$ and a symmetric
  $\phi_\fY$ are given, a {\it symmetrized pullback} of $\phi_\fY$
  along $f$ is the data of an object $\bE_\fX$ and maps $\delta,
  \zeta, \eta$ such that \eqref{eq:symmetric-compatibility-diagram} is
  a morphism of exact triangles. Then $\bE_\fX$ inherits the symmetry
  of $\bE_\fY$, and the five lemma applied to the obvious cohomology
  long exact sequences shows that both $\bF \to \bL_\fX$ and $\bE_\fX
  \to \bF \to \bL_\fX$ are obstruction theories. If $f$ is smooth, our
  convention is to take $\bE_f \coloneqq \bL_f = \Omega_f$ as the
  relative obstruction theory.
\end{definition}

\subsection{}
\label{sec:symmetrized-pullback-construction}

Generally, we try to construct symmetrized pullback in the following
way. Suppose the solid bottom left square in the following
commutative diagram is given:
\begin{equation} \label{eq:symmetrized-pullback-construction}
  \begin{tikzcd}[column sep=4em]
    A \ar[dotted]{r}{\gamma} \ar[dashed,equals]{d} & \cocone(\delta^\vee[1]) \ar[dotted]{r} \ar[dashed]{d} & \cocone(\beta) \ar[dashed]{d} \ar[dotted]{r}{+1} & {} \\
    A \ar{r}{\delta} \ar{d} & f^*\bE_\fY \ar[dashed]{r} \ar{d}{\delta^\vee[1]} & \cone(\delta) \ar[dashed]{d}{\beta} \ar[dashed]{r}{+1} & {}  \\
    0 \ar{r} & t \otimes A^\vee[1] \ar[dashed,equals]{r} & t \otimes A^\vee[1] \ar[dashed]{r}{+1} & {}
  \end{tikzcd}
\end{equation}
For symmetrized pullback we want $A = \bE_f[-1]$, though the following
construction works for any $A$ with amplitude in $[0, 2]$. The given
data induces all dashed arrows, making all but the topmost row into
exact triangles. The octahedral axiom then implies the topmost row can
also be completed into an exact triangle, given by the dotted arrows.

Suppose in addition that $\beta$ is unique (up to isomorphism of
$\cone(\delta)$). Then the morphism $\gamma\colon A \to
\cocone(\delta^\vee[1])$, which a priori came from the octahedral
axiom, must actually be $\beta^\vee[1]$. It follows that the entire
diagram is symmetric across the diagonal after applying $t \otimes
(-)^\vee[1]$. Then set $\bF \coloneqq \cone(\delta)$ and $\bE_\fX
\coloneqq \cocone(\beta)$ to conclude. The desired symmetrized
pullback is $\bE_\fX$ with the natural maps $\bE_\fX \to \bF \to
\bL_\fX$.

To conclude, we constructed a symmetrized pullback assuming that
$\delta^\vee[1] \circ \delta = 0$ and that $\beta$ is unique. In
general, it is unclear how to fulfill these assumptions. We will give
two special settings where they do hold.

\subsection{}

\begin{example}
  Let $\delta\colon A \to f^*\bE_\fY$ be given. Suppose $\Hom(A,
  A^\vee[1]) = 0 = \Hom(A, A^\vee)$. The first equality implies
  $\delta^\vee[1] \circ \delta = 0$. The second equality implies that
  $\beta$ in \eqref{eq:symmetrized-pullback-construction} is unique,
  since any two ways of filling in $\beta$ are homotopic but the
  homotopy lies in $\Hom(A, A^\vee) = 0$. Then the construction in
  \S\ref{sec:symmetrized-pullback-construction} produces the
  symmetrized pullback.

  For instance, suppose that $\fX$ is {\it affine}, $f$ is a (stacky)
  vector bundle, and $\bE_f \coloneqq \bL_f$ is the relative
  obstruction theory. Then $A = \bL_f[-1]$ is locally free of
  amplitude $[1, 2]$, and the desired vanishings occur for degree
  reasons. In fact, in this setting, even the lift
  $\delta\colon A \to f^*\bE_\fY$ of the natural map
  $A \to f^*\bL_\fY$ exists uniquely: the obstructions to existence
  and uniqueness lie in $\Hom(A, \cone(f^*\phi_\fY))$ and
  $\Hom(A, \cone(f^*\phi_\fY)[-1])$ respectively, which also vanish
  for degree reasons. In \cite{KLT}, this observation is used to
  construct symmetrized pullbacks using {\it almost-perfect
    obstruction theories} \cite{Kiem2020}, an \'etale-local version of
  perfect obstruction theory.
\end{example}

\subsection{}

\begin{example} \label{ex:symmetrized-pullback-shifted-cotangent-bundle}
  For this example, which the reader may feel free to skip, we assume
  some knowledge of derived algebraic geometry; see \cite[Ch.
    2.2]{Toen2008} \cite[\S 4.2, 4.3]{Toen2009} for background.
  Suppose that $f\colon \fX \to \fY$ arises as the classical
  truncation of the top row in the following homotopy-Cartesian square
  of derived Artin stacks:
  \begin{equation} \label{eq:symmetrized-pullback-shifted-cotangent-bundle-diagram}
    \begin{tikzcd}
      \fX^\der \ar{r}{f} \ar{d}{\pi} & T^*[-1]\fM^\der \ar{d}{\pi} \\
      \bar\fX^\der \ar{r}{\bar f} & \fM^\der
    \end{tikzcd}
  \end{equation}
  where $\bar f\colon \bar\fX^\der \to \fM^\der$ is a given smooth
  morphism of derived stacks, and $T^*[-1]\fM^\der \coloneqq \tot(t
  \otimes \bL_{\fM^\der}[-1])$ is the $(-1)$-shifted cotangent bundle
  \cite[Ex. 4.3.3]{Toen2009}. Recall that if $\fM^\der$ is a derived
  Artin stack with classical truncation $\fM$ and its natural
  inclusion $i\colon \fM \to \fM^\der$, then:
  \begin{itemize}
  \item the cotangent complex $\bL_{\fM^\der}$ exists in a
    pre-triangulated dg category $L_{\cat{QCoh}}(\fM^\der)$,
    \footnote{If $\fM^\der = i(\fM)$ is actually a classical stack,
      then this is just the usual cotangent complex
      $\bL_\fM \in D\cat{QCoh}(\fM) = L_{\cat{QCoh}}(\fM^\der)$.}
    satisfying the usual properties of cotangent complexes;
  \item $h^k(\bL_i) = 0$ for $k \ge -1$ \cite[Proposition
    1.2]{Schuerg2015} and therefore
    $i^*\bL_{\fM^\der} \to \bL_{\fM}$ is an obstruction theory;
  \item if $f\colon \fM^\der \to \fN^\der$ is a morphism of locally
    finitely presented derived Artin stacks, \footnote{Note that a
      locally finitely presented morphism $f\colon \fM \to \fN$ of
      classical stacks is not necessarily locally finitely presented
      as a morphism $i(f)\colon i(\fM) \to i(\fN)$ of derived stacks.}
    then $\bL_f$ is perfect.
  \end{itemize}
  We assume that all objects and morphisms are
  $\bC^\times$-equivariant for the $\bC^\times$-action which scales
  the (shifted) cotangent direction.

  By definition, $\bL_\pi = t \otimes \pi^*\bL_{\fM^\der}^\vee[1]$,
  and by base change, $\bL_f \cong \pi^* \bL_{\bar f}$. So one can
  form the dashed diagonal maps in
  \begin{equation} \label{eq:symmetrized-section-cosection}
    \begin{tikzcd}
      \bL_f[-1] \ar{d}{\bar\delta} \ar[dashed]{dr}{\delta} \\
      f^*\pi^*\bL_{\fM^\der} \ar{r} & f^*\bL_{T^*[-1]\fM^\der} \ar{r}  \ar[dashed]{dr}[swap]{\delta^\vee[1]} & t \otimes f^*\pi^*\bL_{\fM^\der}^\vee[1] \ar{r}{+1} \ar{d}{\bar\delta^\vee[1]} & {} \\
      {} & {} & t \otimes \bL_f^\vee[2]
    \end{tikzcd}
  \end{equation}
  where the middle row is the exact triangle of cotangent complexes
  for $\pi$, and $\bar\delta$ is the connecting map in the exact
  triangle of cotangent complexes for $\bar f$. Evidently
  $\delta^\vee[1] \circ \delta = 0$, so this $\delta$ can be used in
  the construction of \S\ref{sec:symmetrized-pullback-construction}.
  Cones are functorial in pre-triangulated dg categories and so
  $\beta$ is unique. Classical truncation of the entire diagram
  \eqref{eq:symmetrized-pullback-construction} then produces the
  desired symmetrized pullback.
\end{example}

\subsection{}

\begin{remark} \label{rem:symmetry-is-unnecessary}
  In fact, the uniqueness of $\beta$, and therefore the symmetry of
  $\cocone(\beta)$, is often irrelevant in applications where it is
  enough to construct an obstruction theory $\bE_\fX \to \bL_\fX$ such
  that:
  \begin{itemize}
  \item the restriction to a stable locus $\fX^\st \subset \fX$ is a
    perfect obstruction theory;
  \item the {\it K-theory class} of $\bE_\fX$ is symmetric, i.e.
    $\bE_\fX = -t \bE_\fX^\vee \in K_\sT(\fX)$.
  \end{itemize}
  The construction \eqref{eq:symmetrized-pullback-construction}
  satisfies both properties without assuming uniqueness of $\beta$.
  Namely, the cohomology long exact sequence for the top row of
  \eqref{eq:symmetrized-pullback-construction} gives
  \[ 0 = h^{-2}(A) \to t \otimes h^1(\bF)^\vee \xrightarrow{\sim} h^{-2}(\bE_\fX) \to h^{-1}(A) = 0, \]
  but $h^1(\bF) = h^1(\bL_\fX)$ since $\bF$ is an obstruction theory
  for $\fX$, so indeed $h^{-2}(\bE_\fX)$ vanishes when restricted to
  $\fX^\st$. It is also clear that
  \begin{equation} \label{eq:symmetrized-pullback-k-class}
    \bE_\fX = f^*\bE_\fY - A + t A^\vee \in K_\sT(\fX)
  \end{equation}
  is symmetric. Abusing terminology, hopefully without too much
  confusion, we continue to refer to the constructed
  $\bE_\fX \to \bL_\fX$ as a symmetrized pullback regardless of
  whether $\beta$ is unique.
\end{remark}

\subsection{}

\begin{lemma} \label{lem:symmetrized-pullback-cosection}
  Let $\bE_\fX$ be a symmetrized pullback of $\bE_\fY$ along a smooth
  morphism $f\colon \fX \to \fY$. Then a surjection
  $h^{-1}(\bE_\fY) \twoheadrightarrow \cO_\fY$ induces a surjection
  $h^{-1}(\bE_\fX) \twoheadrightarrow \cO_\fX$.
\end{lemma}

Such surjections are nowhere-vanishing cosections of the obstruction
sheaves, and will be used later to show that certain virtual cycles
vanish via Kiem--Li cosection localization \cite{Kiem2013a}.

\begin{proof}
  Smoothness of $f$ means $A \coloneqq \bE_f[-1]$ has amplitude in
  $[1,2]$ only, so the cohomology long exact sequence for the middle
  row of \eqref{eq:symmetric-compatibility-diagram} or
  \eqref{eq:symmetrized-pullback-construction} gives
  \[ 0 = h^{-1}(A) \to h^{-1}(f^*\bE_\fY) \xrightarrow{\sim} h^{-1}(\bF) \to h^0(A) = 0. \]
  Similarly, the right-most column of
  \eqref{eq:symmetric-compatibility-diagram} or
  \eqref{eq:symmetrized-pullback-construction} gives
  \[ h^{-1}(\bE_\fX) \twoheadrightarrow h^{-1}(\bF) \to h^{-1}(t \otimes A^\vee[1]) \cong t \otimes h^0(A)^\vee = 0. \]
  The desired surjection is then
  $h^{-1}(\bE_\fX) \twoheadrightarrow h^{-1}(\bF) \cong h^{-1}(f^*\bE_\fY) \twoheadrightarrow f^*\cO_\fY = \cO_\fX$.
\end{proof}

\section{Quiver-framed stacks}
\label{sec:quiver-framed-stacks}

\subsection{}

\begin{definition} \label{def:quiver-framed-stack}
  Let $\fN = \bigsqcup_\alpha \fN_\alpha$ be the Vafa--Witten moduli
  stack \eqref{eq:VW-stack} from the introduction. For $k \in \bZ_{>
    0}$, let ${}_k\fN \subset \fN$ be the open locus of $k$-regular
  sheaves, i.e. those $E$ such that $H^i(E(k-i)) = 0$ for all $i > 0$.
  Here are some standard facts \cite[\S 1.7]{Huybrechts2010}. First,
  $k$-regular implies $(k+1)$-regular. So, the functor
  \[ F_k\colon E \mapsto H^0(E \otimes \cO_X(k)) \]
  is exact on $k$-regular sheaves --- in fact, it equals $\chi(E
  \otimes \cO_X(k))$ --- and induces an inclusion $\Hom(E, E) \to
  \Hom(F_k(E), F_k(E))$. Consequently, the quantity $\lambda_k(\alpha)
  \coloneqq \dim F_k(\alpha)$ depends only on the class $\alpha$ of
  $E$. Finally, if $S \subset \fN$ is any finite-type substack, e.g.
  the locus of semistable objects of a given class $\alpha$, there
  exists $k \gg 0$ such that $S \subset {}_k\fN$.

  Let $Q$ be a quiver with no cycles, with edges $Q_1$ and vertices
  $Q_0 = Q_0^o \sqcup Q_0^f$ split into ordinary vertices
  $\boldsymbol{\bullet} \in Q_0^o$ and {\it framing} vertices
  $\blacksquare \in Q_0^f$ such that ordinary vertices have no
  outgoing arrows. For
  $\kappa = (\kappa(v))_{v \in Q_0^o} \in \bZ^{|Q_0^o|}$ and a
  dimension vector $\vec d = (d_v)_{v \in Q_0^f}$, let
  $\fN^{Q(\kappa)}_{\alpha,\vec d}$ be the moduli stack of triples
  $(E, \vec V, \vec \rho)$ where:
  \begin{itemize}
  \item
    $[E] \in {}_\kappa\fN_\alpha \coloneqq \bigcap_{v \in Q_0^o} {}_{\kappa(v)}\fN_\alpha$;
  \item $\vec V = (V_v)_{v \in Q_0^f}$ where $V_v$ is a
    $d_v$-dimensional vector space; set
    $V_v \coloneqq F_{\kappa(v)}(E)$ for $v \in Q_0^o$;
  \item $\vec \rho = (\rho_e)_{e \in Q_1}$ are linear maps between the
    $V_v$.
  \end{itemize}
  A morphism between two triples $(E, \vec V, \vec \rho)$ and
  $(E', \vec V', \vec \rho')$ consists of a morphism $E \to E'$,
  inducing morphisms $V_v \to V'_v$ for all $v \in Q_0^o$, along with
  morphisms $V_v \to V_v'$ for $v \in Q_0^f$ intertwining $\vec\rho$
  and $\vec\rho'$.
\end{definition}

\subsection{}

\begin{remark} \label{rem:rigidification}
  Every object in a $\bC$-linear category has an action by the group
  $\bC^\times$ of scaling automorphisms. Our convention is that a
  moduli stack $\fX = \bigsqcup_\beta \fX_\beta$ of such objects has
  already been {\it $\bC^\times$-rigidified} \cite[\S
  5.1]{Abramovich2003}, meaning that this $\bC^\times$ has been
  quotiented away from all stabilizer groups. Let
  $\fX^{\unrig} = \bigsqcup_\beta \fX_\beta^\unrig$ be the
  unrigidified moduli stack, with $\fX_0 = \{0\}$ containing only the
  zero object, so that
  \[ \rig_\beta\colon \fX_\beta^{\unrig} \to \fX_\beta \]
  is a principal $[\pt/\bC^\times]$-bundle for all $\beta \neq 0$.
  When universal families exist on $\fX^{\unrig}$, they have
  $\bC^\times$-weight one and so may not descend to $\fX$; see
  \cite[\S 4.6]{Huybrechts2010}.

  On $\fN^{Q(\kappa),\unrig}_{\alpha,\vec d}$, for each vertex $v$,
  let $\cV_v$ be the universal bundle of $V_v$. Indeed, by
  $k$-regularity, $\cV_v$ for $v \in Q_0^f$ is a vector bundle of rank
  $\lambda_{\kappa(v)}(\alpha)$. Then
  \begin{equation} \label{eq:quiver-stack-projection}
    \Pi_\alpha\colon \fN^{Q(\kappa),\unrig}_{\alpha,\vec d} = \bigg[\bigoplus_{[v \to v'] \in Q_1} \cV_v^\vee \otimes \cV_{v'} / \prod_{v \in Q_0^f} \GL(d_v)\bigg] \to {}_\kappa\fN_\alpha^{\unrig}
  \end{equation}
  is a stacky vector bundle. In particular, $\Pi_\alpha$ is smooth.
  Clearly $\Pi_\alpha$ is equivariant for the diagonal action of
  $\bC^\times$ on the $\GL(d_v)$ and the stabilizer groups of
  ${}_\kappa\fN_\alpha^{\unrig}$, so it descends to a projection
  $\fN^{Q(\kappa)}_{\alpha,\vec d} \to {}_\kappa\fN_\alpha$ which we
  continue to denote $\Pi_\alpha$. It has the same fibers as
  \eqref{eq:quiver-stack-projection}: while the individual $\cV_v$ do
  not descend, $\bC^\times$-weight-zero combinations like $\cV_v^\vee
  \otimes \cV_{v'}$ do descend. \footnote{In fact, the image of
  $\rig_\beta^*\colon \cat{QCoh}(\fX_\beta) \to
  \cat{QCoh}(\fX_\beta^{\unrig})$ consists {\it exactly} of those
  quasi-coherent modules with $\bC^\times$-weight zero \cite[Prop.
    4.10]{Bergh2021}.}

  All enumerative invariants live on open loci of $\fN_\alpha$, but
  the operations we will do on obstruction theories in
  \S\ref{sec:quiver-stacks-obstruction-theory-alternate} take place
  naturally on $\fN_\alpha^{\unrig}$.
\end{remark}

\subsection{}
\label{sec:quiver-stacks-obstruction-theory}

The symmetric obstruction theory on ${}_\kappa \fN_\alpha$ admits a
symmetrized pullback along $\Pi_\alpha$ following
Example~\ref{ex:symmetrized-pullback-shifted-cotangent-bundle}. First,
replace all moduli stacks with their derived enhancements, and let
\[ \bar\fN_\alpha^{\der} \coloneqq \{\det \bar E = L\} \subset \bar\fM_\alpha^{\der} \]
be the derived moduli stack of coherent sheaves $\bar E$ on $S$ with
fixed determinant $L$. The construction of
Definition~\ref{def:quiver-framed-stack} applies equally well to
$\bar\fN_\alpha^{\der}$. Since
$H^0(E \otimes \cO_X(k)) = H^0(\bar E \otimes \cO_S(k))$, there is a
homotopy-Cartesian square
\[ \begin{tikzcd}
    \fN_{\alpha,\vec d}^{Q(\kappa),\der} \ar{r}{\Pi_\alpha} \ar{d}{\pi} & {}_\kappa \fN_\alpha^{\der} \ar{d} \\
    \bar\fN_{\alpha,\vec d}^{Q(\kappa),\der} \ar{r}{\bar\Pi_\alpha} & {}_\kappa \bar \fN_\alpha^{\der}
  \end{tikzcd} \]
of the form
\eqref{eq:symmetrized-pullback-shifted-cotangent-bundle-diagram}. Note
that ${}_\kappa\fN_\alpha$ is an open locus in $\fN_\alpha^{\der} =
T^*[-1]\bar\fN_\alpha^{\der}$ and therefore has the same cotangent
complex (after restriction). Since the symmetric obstruction theory
\eqref{eq:VW-reduced-obstruction-theory} on the classical stack $\fN$
is the classical truncation of $\bL_{\fN^{\der}}$ \cite[\S
  1.7]{Tanaka2020} \cite[Proposition 5.2]{Schuerg2015}, the
construction of
Example~\ref{ex:symmetrized-pullback-shifted-cotangent-bundle} indeed
provides the desired symmetrized pullback.

\subsection{}
\label{sec:quiver-stacks-obstruction-theory-alternate}

Alternatively, the obstruction theory on $\fN^{Q(\kappa)}_{\alpha,\vec
  d}$ may be constructed without resorting to derived algebraic
geometry. The main question is how to explicitly construct a lift
\[ \begin{tikzcd}
    {} & \pi^*\bar\Pi_\alpha^* \bE_{{}_\kappa \bar\fN_\alpha} \ar{d}{\pi^*\Pi_\alpha^*\phi_\alpha} \\
    \bL_{\Pi_\alpha}[-1] \ar{r} \ar[dashed]{ur}{\bar\delta} & \pi^*\bar\Pi_\alpha^* \bL_{{}_\kappa \bar\fN_\alpha}
  \end{tikzcd} \]
of the connecting map in the exact triangle of relative cotangent
complexes. Here $\phi_\alpha$ is the natural obstruction theory
\cite[\S 5.6]{Tanaka2020} given by the traceless part of the Atiyah
class of the universal sheaf $\bar\cE$, or equivalently $\bar\cE(k)$,
on ${}_\kappa \bar\fN_\alpha^{\unrig}$. This $\bar\delta$ may then be
used to define the dashed diagonal maps in
\[ \begin{tikzcd}
  \bL_{\Pi_\alpha}[-1] \ar{d}{\bar\delta} \ar[dashed]{dr}{\delta} \\
  \pi^*\bar\Pi_\alpha^*\bE_{{}_\kappa\bar\fN_\alpha} \ar{r} & \Pi_\alpha^*\bE_{{}_\kappa\fN_\alpha} \ar{r} \ar[dashed]{dr}[swap]{\delta^\vee[1]} & t \otimes \pi^*\bar\Pi_\alpha^*\bE_{{}_\kappa\bar\fN_\alpha}^\vee[1] \ar{r}{+1} \ar{d}{\bar\delta^\vee[1]} & {} \\
  {} & {} & t \otimes \bL_{\Pi_\alpha}^\vee[2],
\end{tikzcd} \]
cf. \eqref{eq:symmetrized-section-cosection}, where the middle row is
given by the exact triangle \eqref{eq:VW-reduced-obstruction-theory}
defining the obstruction theory on $\fN$. Thus we obtain the desired
map $\delta\colon \bL_{\Pi_\alpha}[-1] \to \Pi_\alpha^* \bE_{{}_\kappa
  \fN_\alpha}$ such that $\delta^\vee[1] \circ \delta = 0$.

We construct $\bar\delta$. Note that $\bL_{\Pi_\alpha} =
\pi^*\bL_{\bar\Pi_\alpha}$. The universal bundles $(\cV_v)_{v \in
  Q_0^o}$ on $\fN^{Q(\kappa),\unrig}_{\alpha,\vec d}$ are pulled back
from similar ones on ${}_\kappa \bar\fN_\alpha^{\unrig}$, which induce
a classifying morphism
\[ c_{\cV}\colon {}_\kappa\bar\fN_\alpha^{\unrig} \to [\pt/\GL(\lambda_\kappa(\alpha))], \qquad \GL(\lambda_\kappa(\alpha)) \coloneqq \prod_{v \in Q_0^o} \GL(\lambda_{\kappa(v)}(\alpha)). \]
Since the stacky vector bundle \eqref{eq:quiver-stack-projection} can
be equally well constructed using the universal bundle on
$[\pt/\GL(\lambda_{\kappa(v)}(\alpha))]$ in place of $\cV_v$ for $v
\in Q_0^o$, there is the following natural Cartesian square and
induced commutative square of cotangent complexes
\[ \begin{tikzcd}[row sep=small, column sep=small]
  \bar\fN^{Q(\kappa),\unrig}_{\alpha,d} \ar{r}{\bar\Pi_\alpha} \ar{dd} & {}_\kappa \bar\fN_\alpha^{\unrig} \ar{dd}{c_{\cV}} & & \bL_{\bar\Pi}[-1] \ar{r} \ar{dd}{\rotatebox{90}{$\sim$}} & \bar\Pi^*\bL_{[\pt/\GL(\lambda_\kappa(\alpha))]} \ar{dd}{c_{\cV}^*} \\
  & & \leadsto \\
  {\big[\bigoplus\limits_{v \to v'} \cV_v^\vee \otimes \cV_{v'} / \!\prod\limits_{v \in Q_0^f} \GL(d_v)\big]} \ar{r}{\bar\Pi} & {}[\pt/\GL(\lambda_\kappa(\alpha))] & & \bL_{\bar\Pi_\alpha}[-1] \ar{r} & \bar\Pi_\alpha^*\bL_{{}_\kappa\bar\fN_\alpha^{\unrig}},
\end{tikzcd} \]
After rigidification, this induces the dashed diagonal arrow in the
commutative diagram
\[ \begin{tikzcd}
    & \bar\Pi_\alpha^* c_\cV^* \bL_{[\pt/\PGL(\lambda_\kappa(\alpha))]} \ar{d} \ar[dotted]{r} & \bar\Pi_\alpha^* \bE_{{}_\kappa\bar\fN_\alpha} \ar{dl}{\bar\Pi_\alpha^*\phi_\alpha} \\
    \bL_{\bar\Pi_\alpha}[-1] \ar[dashed]{ur} \ar{r} & \bar\Pi_\alpha^* \bL_{{}_\kappa\bar\fN_\alpha}.
  \end{tikzcd} \]
The middle vertical arrow is the traceless part of the cotangent
complex map for $c_\cV$, which can be identified with the traceless
part of the sum of Atiyah classes of $\cV_v$ for $v \in Q_0^o$ (cf.
\cite[Remark A.1]{Schuerg2015}). By construction there are maps $\cV_v
\to \bar \cE(k)$, so functoriality of the (traceless part of the)
Atiyah class gives the dotted horizontal arrow. Thus we get the
desired map $\bar\delta$ by pulling back the composition
$\bL_{\bar\Pi_\alpha}[-1] \to
\bar\Pi_\alpha^*\bE_{{}_\kappa\bar\fN_\alpha}$ along $\pi$.

It is unclear from the classical construction whether the resulting
$\beta$ in \eqref{eq:symmetrized-pullback-construction} is unique, but
for our purposes this is irrelevant by
Remark~\ref{rem:symmetry-is-unnecessary}.

\subsection{}

\begin{definition} \label{def:stability-conditions}
  For a class $\alpha = (r, c_1(L), c_2)$, let
  \[ P_\alpha(n) \coloneqq \chi(\alpha(n)) = r(\alpha) n^{\dim \alpha} + O(n^{\dim \alpha - 1}) \]
  be the Hilbert polynomial of the class $\alpha$. Then
  \[ \tau(\alpha) \coloneqq P_\alpha/r(\alpha) \]
  is a monic polynomial. Following \cite[Definition 7.7]{Joyce2021},
  put a total order $\le$ on monic polynomials: $f \le g$ if either
  $\deg f > \deg g$, or $\deg f = \deg g$ and $f(n) \le g(n)$ for all
  $n \gg 0$. Then $\tau$ is a stability condition on $\fN$ in Joyce's
  sense \cite[Definition 3.1]{Joyce2021}, \footnote{While Joyce works
  with the more general notion of {\it weak} stability condition,
  Gieseker stability is genuinely a stability condition. This is
  important in the verification that
  \eqref{eq:quiver-stack-stability-condition} is a stability
  condition.} and we call it {\it Gieseker stability} for obvious
  reasons. We say the polarization $\cO_S(1)$ is {\it generic} for
  $\alpha$ if
  \begin{equation} \label{eq:generic-polarization}
    \tau(\beta) = \tau(\alpha) \implies \beta = \text{constant} \cdot \alpha.
  \end{equation}

  The quantity
  $\tilde r(\alpha) \coloneqq (\dim \alpha)! \cdot r(\alpha)$ is a
  positive integer when $\alpha \neq 0$, and if
  $\tau(\alpha) = \tau(\beta)$ then
  $\tilde r(\alpha + \beta) = \tilde r(\alpha) + \tilde r(\beta)$.
  Following \cite[(5.13)]{Joyce2021}, given
  $\vec\mu \in \bR^{|Q_0^o|}$ define
  \begin{equation} \label{eq:quiver-stack-stability-condition}
    \tau_{\vec\mu}(\alpha, \vec d) \coloneqq \begin{cases} \left(\tau(\alpha), \vec\mu \cdot \vec d/\tilde r(\alpha)\right) & \alpha \neq 0 \\ \left(\infty, \vec\mu \cdot \vec d / |\vec d|\right) & \alpha = 0, \; \vec\mu\cdot\vec d > 0 \\ \left(-\infty, \vec\mu \cdot \vec d / |\vec d|\right) & \alpha = 0, \; \vec\mu\cdot \vec d \le 0. \end{cases}
  \end{equation}
  Use the total order given by: $(a, b) \le (a', b')$ if either $a \le
  a'$, or $a=a'$ and $b \le b'$. Then $\tau_{\vec\mu}$ is a stability
  condition on $\fN^{Q(\kappa)}$, using the additivity of the
  functions $\tilde r$ and $\vec\mu \cdot (-)$. Clearly
  $\tau_{\vec\mu}$-semistability of $(E, \vec V, \vec \rho)$ implies
  $\tau$-semistability of $E$.
\end{definition}

\subsection{}

\begin{lemma} \label{lem:quiver-stack-fixed-locus-is-proper}
  If $\fN_{\alpha,\vec d}^{Q(\kappa)}$ has no strictly
  $\tau_{\vec\mu}$-semistables, then the fixed locus
  $(\fN_{\alpha,\vec d}^{Q(\kappa),\sst})^\sT$ is proper.
\end{lemma}

\begin{proof}
  View $X$ as an open locus in the projective $3$-fold
  $\hat X \coloneqq \bP(K_S \oplus \cO_S)$, and let
  $\hat \fN_{\alpha,\vec d}^{Q(\kappa)}$ and
  $\hat \fM_{\alpha,\vec d}^{Q(\kappa)}$ be the analogue of
  $\fN_{\alpha,\vec d}^{Q(\kappa)}$ and
  $\fM_{\alpha,\vec d}^{Q(\kappa)}$ respectively on $\hat X$ instead
  of $X$. It suffices to prove that
  $\hat \fM_{\alpha,\vec d}^{Q(\kappa),\sst}$ is proper, because of
  the chain of closed immersions
  \[ \left(\fN_{\alpha,\vec d}^{Q(\kappa),\sst}\right)^\sT \subset \left(\hat \fN_{\alpha,\vec d}^{Q(\kappa),\sst}\right)^\sT \subset \hat \fN_{\alpha,\vec d}^{Q(\kappa),\sst} \subset \hat \fM_{\alpha,\vec d}^{Q(\kappa),\sst}. \]
  To get properness of $\hat \fM_{\alpha,\vec d}^{Q(\kappa),\sst}$, we
  invoke the results of \cite[\S 7.4]{Joyce2021} in order to avoid
  lengthy technicalities. \footnote{We give machinery for a more
  self-contained and elementary proof in \cite[\S 4.3]{KLT}, using the
  valuative criterion and Langton's elementary modifications.} The
  machinery there \cite[Condition 7.17]{Joyce2021} requires either
  $\dim \alpha = 0, 1, 3$ \cite[Proposition 7.19]{Joyce2021}, which we
  do not satisfy, or the polarization for Gieseker stability to be a
  rational K\"ahler class \cite[Proposition 7.18]{Joyce2021}, which we
  do satisfy.
\end{proof}

\section{Refined pairs invariants}
\label{sec:joyce-song}

\subsection{}

We begin with some generalities on equivariant K-theory. Let $M$ be a
finite-type scheme with an action by a torus $\sT$, and suppose $M$
has the $\sT$-equivariant resolution property \cite[\S 2]{Totaro2004};
we say that {\it equivariant localization holds} for such $M$.
\footnote{Recent technical advances make the resolution property
  unnecessary \cite[\S 5]{Aranha2022}, at least for a certain type of
  Borel--Moore homology, and in full (classical) generality $M$ is
  allowed to be any finite-type Deligne--Mumford stack with
  $\sT$-action. } Specifically, for us, $M$ will be a quasi-projective
scheme. Let $K_\sT(M)$, resp. $K_\sT^\circ(M)$, be the Grothendieck
group of $\sT$-equivariant coherent sheaves on $M$, resp. perfect
complexes (equivalently, by the resolution property, vector bundles),
on $M$. Both are modules for $\bk_\sT \coloneqq K_\sT(\pt)$. Set
\begin{equation} \label{eq:localized-base-ring}
  \bk_{\sT,\loc} \coloneqq \bk_\sT[(1 - w)^{-1} : w \neq 1],
\end{equation}
where $w$ ranges over non-trivial weights of $\sT$. Let
$\iota\colon M^\sT \hookrightarrow M$ be the inclusion of the
$\sT$-fixed locus. A perfect obstruction theory $\bE_M$ on $M$ gives:
\begin{itemize}
\item virtual structure sheaves $\cO_M^\vir \in K_\sT(M)$ and
  $\cO_{M^\sT}^\vir \in K_\sT(M^\sT)$ \cite{Ciocan-Fontanine2009,
    Fantechi2010}, the latter because the fixed part
  $\bE_{M^\sT} \coloneqq (\iota^*\bE_M)^\sT$ is a perfect obstruction
  theory on $M^\sT$ \cite{Graber1999};
\item on
  $K_\sT(M)_\loc \coloneqq K_\sT(M) \otimes_{\bk_\sT} \bk_{\sT,\loc}$,
  the equivariant localization formula \cite{Qu2018}
  \begin{equation} \label{eq:equivariant-localization}
    \cO_M^\vir = \iota_* \frac{\cO_{M^\sT}^\vir}{\se(N^\vir_\iota)} \in K_\sT(M)_\loc
  \end{equation}
  where $\se$ is the K-theoretic Euler class and $(N^\vir_\iota)^\vee$
  is the non-fixed part of $\iota^*\bE_M$.
\end{itemize}
Using the resolution property, $N^\vir_\iota$ can be presented as a
two-term complex of vector bundles $[E_0 \to E_1]$ all of whose
$\sT$-weights are non-trivial, and, for such complexes, the
K-theoretic Euler class $\se$ is
\begin{equation} \label{eq:k-theoretic-euler-class}
  \se(E) \coloneqq \se(E_0)/\se(E_1) \in K_\sT(M)_\loc, \qquad \se(E_k) \coloneqq \sum_{i \ge 0} (-1)^i \wedge^i\! E_k^\vee.
\end{equation}
When there is no ambiguity, we sometimes omit the subscripts on
$\cO^\vir$ and $N^\vir$.

\subsection{}
\label{sec:symmetrized-virtual-localization}

Suppose that, in addition, the K-theory class of $\bE_M$ is symmetric.
Then we also get:
\begin{itemize}
\item a canonical square root of the restriction of $K^\vir_M
  \coloneqq \det \bE_M$ to $M^\sT$, given by
  \begin{equation} \label{eq:virtual-half-canonical}
    K^{\vir,\frac{1}{2}}_M \coloneqq t^{-\frac{1}{2} \rank \bE_{>0}} \det \bE_{>0}
  \end{equation}
  where $\bE_{>0}$ is the summand of $\iota^*\bE_M$ with positive
  $\sT$-weight \cite[Proposition 2.6]{Thomas2020}; \footnote{Our
    conventions differ slightly from those of Thomas because our $t$
    is his $t^{-1}$.}
\item a {\it symmetrized} virtual cycle (by twisting
  \eqref{eq:equivariant-localization} by $K^{\vir,\frac{1}{2}}_M$)
  \begin{equation} \label{eq:symmetrized-virtual-cycle}
    \hat\cO_M^\vir \coloneqq \iota_* \left(\frac{\cO_{M^\sT}^\vir}{\se(N^\vir_\iota)} \otimes K^{\vir,\frac{1}{2}}_M\right) \in K_\sT(M)_\loc[t^{\pm \frac{1}{2}}].
  \end{equation}
\end{itemize}
Unlike $\iota^* K_M^{\vir}$, neither $K^{\vir}_M$ nor $\det
N^\vir_\iota$ is guaranteed to admit a square root. However, in
certain situations, e.g. \S\ref{sec:master-space-relation-proof}, 
\[ (N^\vir_\iota)^\vee = F - t F^\vee + \cdots \in K^\circ_\sT(M^\sT) \]
for some K-theory class $E \coloneqq F - t F^\vee$ whose
contribution\footnote{Since $\iota^*\bE_M = (N_\iota^\vir)^\vee +
\bE_{M^\sT}$, its determinant factorizes as $K_M^{\vir} =
\det(N_\iota^\vir)^\vee \det(\bE_{M^{\sT}})$.} to
\eqref{eq:virtual-half-canonical} we wish to isolate. If $E_{>0}
\coloneqq F_{>0} - t (F_{\le 0})^\vee$ is the summand of $E$ with
positive $\sT$-weight, then
\begin{equation} \label{eq:half-canonical-for-balanced-classes}
  t^{-\frac{1}{2} \rank E_{>0}} \det E_{>0} = t^{-\frac{1}{2} \rank F} \det F.
\end{equation}
Accordingly, define the {\it symmetrized} K-theoretic Euler class
$\hat\se(E) \coloneqq t^{-\frac{1}{2} \rank F} \det F \otimes \se(E)$.
Sometimes we write $\hat \se(F^\vee)/\hat \se(t^{-1} F) \coloneqq \hat
\se(E^\vee)$ --- a mild but suggestive abuse of notation.

\subsection{}

A useful result of Thomas is that the K-theoretic virtual cycle
$\cO^\vir_M$ depends only on the K-theory class of the perfect
obstruction theory $\bE_M$ \cite{Thomas2022}. Then clearly the same is
true for $\hat\cO^\vir_M$. We will use this fact to simplify some
calculations in the remainder of this paper, in situations where $M$
has two perfect obstruction theories constructed in different ways.
Rather than checking that they are isomorphic, which often involves
some annoying diagram-chasing, it suffices to check that they are
equal in K-theory and therefore induce the same (symmetrized)
K-theoretic virtual cycle.

\subsection{}

\begin{definition} \label{def:refined-pairs-invariant}
  Fix $k \gg 0$ and consider the quiver
  \begin{equation} \label{eq:pairs-quiver}
    Q(k) \coloneqq
    \begin{tikzpicture}
      \coordinate[vertex, label=right:$F_k(E)$] (E);
      \coordinate[framing, left=of E, label=left:$V$] (V);
      \draw (E) -- node[label=above:$\rho$]{} (V);
    \end{tikzpicture}
  \end{equation}
  for class $\alpha$ and framing dimension $d = 1$. On
  $\fN^{Q(k)}_{\alpha,1}$, take the symmetric obstruction theory
  obtained by symmetrized pullback along $\Pi_\alpha$, and the
  stability condition $\tau_{\vec\mu}$ from
  \eqref{eq:quiver-stack-stability-condition} with $\mu = 1$. There
  are no strictly semistables \cite[Example 5.6]{Joyce2021}, so the
  semistable locus is a quasi-projective scheme. The {\it refined
    pairs invariant}
  \begin{equation} \label{eq:pairs-invariant}
    \tilde\VW_\alpha(k, t) \coloneqq \chi\left(\fN^{Q(k),\sst}_{\alpha,1}, \hat\cO^\vir\right) \in \bk_{\sT,\loc}[t^{\pm \frac{1}{2}}]
  \end{equation}
  is well-defined by $\sT$-equivariant localization
  (Lemma~\ref{lem:quiver-stack-fixed-locus-is-proper}).
\end{definition}

\subsection{}

\begin{remark} \label{rem:our-vs-JS-pairs}
  The refined pairs invariant in \cite[\S 5]{Thomas2020} is defined
  like in \eqref{eq:pairs-invariant} but using a moduli scheme of
  stable Joyce--Song pairs and its Joyce--Song perfect obstruction
  theory. These can be matched with
  Definition~\ref{def:refined-pairs-invariant}.

  First, recall that a Joyce--Song pair $(E, s)$ is a point
  $[E] \in \fN$ along with a non-zero section
  $s\colon \cO_X(-k) \to E$, and it is stable if and only if:
  \begin{itemize}
  \item $E$ is $\tau$-semistable;
  \item if $s$ factors through $0 \neq E' \subsetneq E$, then
    $\tau(E') < \tau(E)$.
  \end{itemize}
  Clearly the data $(E, s)$ is equivalent to the data $(E, V, \rho)$
  where $\dim V = 1$, and one can check \cite[Example 5.6]{Joyce2021}
  that $\tau_1$-stability is equivalent to Joyce--Song stability.
  Hence $\fN^{Q(k),\sst}_{\alpha,1}$ is indeed a moduli scheme of
  stable Joyce--Song pairs.

  Second, by \cite[(6.2)]{Tanaka2017}, the standard Joyce--Song-style
  symmetric perfect obstruction theory on $\fN^{Q(k)}_{\alpha,1}$ is
  given by $R\Hom_X(I^\bullet, I^\bullet)_\perp$ (up to a dual and
  shift) at the point $I^\bullet \coloneqq [\cO(-k) \xrightarrow{s} E]
  \in D^b\cat{Coh}_\sT(X)$, where
  \[ R\Hom_X(I^\bullet, I^\bullet) \cong R\Hom_X(I^\bullet, I^\bullet)_\perp \oplus H^*(\cO_X) \oplus H^{\ge 1}(\cO_S) \oplus H^{\le 1}(K_S)[-1] \]
  in the derived category. The hard work in \cite{Tanaka2020}, see
  also \cite[Remark 6.3]{Tanaka2017}, may be adapted to show that this
  holds in families. Using the obvious exact triangles and
  \eqref{eq:VW-reduced-obstruction-theory}, in K-theory
  \begin{alignat}{3} 
    R\Hom_X(I^\bullet, I^\bullet)_\perp
    &= R\Hom_X(E, E) && - R\Hom_X(\cO_X(-k), E) - H^{\ge 1}(\cO_S) \nonumber \\
    {} & {} && - R\Hom_X(E, \cO_X(-k)) + H^{\le 1}(K_S) \nonumber \\
    &= R\Hom_X(E, E)_\perp && - \left(R\Hom_X(\cO_X(-k), E) - H^0(\cO_S)\right) \label{eq:pairs-obstruction-theory} \\
    {} & {} && + \left(R\Hom_X(\cO_X(-k), E) - H^0(\cO_S)\right)^\vee \otimes t^{-1}. \nonumber
  \end{alignat}
  This computation clearly also works over families. Comparing with
  \eqref{eq:symmetrized-pullback-k-class}, and using the fact that
  $R^{>0}\Hom_X(\cO_X(-k), E) = 0$, this is the same K-theory class as
  the (dual of a shift of the) obstruction theory obtained from
  symmetrized pullback, and therefore induces the same $\hat\cO^\vir$.
\end{remark}

\subsection{}

\begin{proposition}[{\cite[Proposition 5.5]{Thomas2020}}] \label{prop:symmetrized-integration-on-projective-bundle}
  If $\alpha$ has no strictly semistables, then
  \eqref{eq:semistable-VW-invariant} becomes
  \[ \tilde\VW_\alpha(k, t) = [\lambda_k(\alpha)]_t \cdot \VW_\alpha(t). \]
  The projection
  $\Pi_\alpha\colon \fN^{Q(k),\sst}_{\alpha,1} \to \fN^\sst_\alpha$ is
  a $\bP^{\lambda_k(\alpha)-1}$-bundle, and $\VW_\alpha(t)$ is given
  by \eqref{eq:stable-VW-invariant}.
\end{proposition}

\begin{proof}[Proof sketch.]
  Induct on the rank $r$ of $\alpha = (r, c_1, c_2)$. If
  \eqref{eq:semistable-VW-invariant} had a non-zero contribution
  indexed by $\alpha_1, \ldots, \alpha_n$ with $n > 1$, then the
  non-vanishing of the $\VW_{\alpha_i}(t)$ (which equal
  \eqref{eq:stable-VW-invariant} by the induction hypothesis) would
  imply the $\fN_{\alpha_i}^{\sst}$ are non-empty, and picking an
  element $E_i$ in each would give a strictly semistable
  $[E_1 \oplus \cdots \oplus E_n] \in \fN_\alpha^\sst$, a
  contradiction. So only the $n=1$ term in
  \eqref{eq:semistable-VW-invariant} is present.

  It is also easy to check that, if $E$ is stable, then $(E, V, \rho)$
  is (semi)stable if and only if $\rho \neq 0$, so indeed $\Pi_\alpha$
  is a $[\Hom(V, E(k)) \setminus \{0\}/\bC^\times]$-bundle.

  Finally, note that the first line of
  \eqref{eq:pairs-obstruction-theory} is the part of the obstruction
  theory on $\fN^{Q(k)}_{\alpha,1}$ compatible with the one on
  $\fN_\alpha$, and the second line is the vector bundle
  $\Omega_{\Pi_\alpha} \otimes t^{-1}$, so by virtual pullback
  \cite{Manolache2012} or otherwise,
  \[ \hat\cO^\vir_{\fN^{Q(k),\sst}_{\alpha,1}} = \Pi_\alpha^* \hat\cO^\vir_{\fN^{\sst}_\alpha} \otimes \se\left(\Omega_{\Pi_\alpha} \otimes t^{-1}\right) \otimes t^{-\frac{1}{2} \dim \Pi_\alpha} K_{\Pi_\alpha} \]
  where $K_{\Pi_\alpha} \coloneqq \det \Omega_{\Pi_\alpha}$ is the
  relative canonical. The $t^{-\frac{1}{2} \dim \Pi_\alpha}
  K_{\Pi_\alpha}$ term is the contribution of $\Omega_{\Pi_\alpha} -
  t\Omega_{\Pi_\alpha}^\vee$ to
  $K^{\vir,\frac{1}{2}}_{\fN_{\alpha,1}^{Q(k),\sst}}$; see
  \eqref{eq:virtual-half-canonical} and
  \eqref{eq:half-canonical-for-balanced-classes}. Applying
  $\Pi_{\alpha *}$ and using the projection formula, it remains to
  compute $\Pi_{\alpha *} \left(\se\left(\Omega_{\Pi_\alpha} \otimes
  t^{-1}\right) \otimes K_{\Pi_\alpha}\right)$. This is an alternating
  sum of relative cohomology bundles $R^i\Pi_{\alpha
    *}\Omega^j_{\Pi_\alpha}$, which are all canonically trivialized by
  powers of the (fiberwise) hyperplane class. So it is enough to
  compute on fibers:
  \begin{align*}
    \chi\left(\bP^N, \se\left(\Omega_{\bP^N} \otimes t^{-1}\right) \otimes K_{\bP^N}\right) t^{-\frac{N}{2}}
    &= (-1)^N t^{-\frac{N}{2}} \sum_{i,j=0}^N (-1)^{i+j} \dim H^{i,j}(\bP^N) t^j \\
    &= [N+1]_t \in \bZ[t^{\pm \frac{1}{2}}],
  \end{align*}
  using that the modules $H^{i,j}(\bP^N) \subset H^{i+j}(\cO)$ are
  $\sT$-equivariantly trivial by Hodge theory. This computation
  motivates the sign $(-1)^{N-1}$ in the definition
  \eqref{eq:quantum-integer} of $[N]_t$.
\end{proof}

\subsection{}
\label{sec:U-pairs-invariants}

To be precise, our $\tilde\VW_\alpha$ is the so-called refined
$\mathrm{SU}(r)$ pairs invariant. The $\mathrm{S}$ here corresponds to
fixing $\det \bar E = L$ and $\tr \phi = 0$, i.e. using the moduli
stack $\fN_\alpha \subset \fM_\alpha$ instead of the entire
$\fM_\alpha$. There is an analogous $\mathrm{U}(r)$ pairs invariant
\[ \chi\left(\fM_{\alpha,1}^{Q(k),\sst}, \hat\cO^\vir\right) \]
where the symmetric obstruction theory on $\fM_\alpha$ is also given
by \eqref{eq:VW-reduced-obstruction-theory} except the subscript $0$
on the latter two terms now only removes a copy of $H^0(\cO_S)$ and
$H^2(K_S) \cong H^0(\cO_S)^\vee$ respectively, cf. \cite[\S
  4.1]{Tanaka2020}. This $\mathrm{U}(r)$ pairs invariant agrees with
$\tilde \VW_\alpha$ if $H^1(\cO_S) = H^2(\cO_S) = 0$. Otherwise:
\begin{itemize}
\item if $H^2(\cO_S) \neq 0$, then the summand $H^2(\cO_S)$ in
  $R\Hom_S(\bar E, \bar E)_0$ forms a trivial sub-bundle in the
  obstruction sheaf on $\fM$, so $\hat\cO^\vir = 0$;
\item if $H^1(\cO_S) \neq 0$, then $\Pic_0(S)$ is non-trivial and acts
  on $\fM$ by tensoring. This action has no fixed points and defines a
  $\dim H^1(\cO_S)$-dimensional space of nowhere-vanishing sections of
  the tangent sheaf on $\fM$, which is equivalently a $\dim
  H^1(\cO_S)$-dimensional space of nowhere-vanishing ``$t$-twisted''
  cosection (i.e. landing in $t \otimes \cO$ instead of $\cO$) of the
  obstruction sheaf on $\fM$ by the symmetry of the obstruction
  theory. Lemma~\ref{lem:symmetrized-pullback-cosection} lifts this to
  a $\dim H^1(\cO_S)$-dimensional space of nowhere-vanishing
  ``$t$-twisted'' cosection of the obstruction sheaf on $\fM^{Q(k)}$,
  so $\hat\cO^\vir$ contains a factor of $(1 - t)^{\dim H^1(\cO_S)}$
  by cosection localization \cite{Kiem2013a}. Thus it vanishes modulo
  $(1 - t)^{\dim H^1(\cO_S)}$.
\end{itemize}
These vanishings are important motivation for using
$\fN_\alpha \subset \fM_\alpha$ to define $\tilde\VW_\alpha$, and will
also be important in the master space calculation in
\S\ref{sec:mixed-fixed-locus-cases}.

\section{A master space calculation}
\label{sec:master-space-calculation}

\subsection{}

\begin{definition} \label{def:master-space}
  Let $k_1, k_2 \gg 0$ and, as in \cite[Definition 9.4]{Joyce2021},
  consider the quiver
  \begin{equation} \label{eq:master-space-quiver}
    \tilde Q(k_1, k_2) \coloneqq \begin{tikzpicture}[node distance=0.3cm and 1cm]
      \coordinate[framing, label=below:$V_3$] (V3);
      \coordinate[framing, above right=of V3, label=above:$V_1$] (V1);
      \coordinate[framing, below right=of V3, label=below:$V_2$] (V2);
      \coordinate[vertex, right=of V1, label=right:$F_{k_1}(E)$] (F1);
      \coordinate[vertex, right=of V2, label=right:$F_{k_2}(E)$] (F2);
      \draw[->] (V3) -- node[label=above:$\rho_3$]{} (V1);
      \draw[->] (V1) -- node[label=above:$\rho_1$]{} (F1);
      \draw[->] (V3) -- node[label=below:$\rho_4$]{} (V2);
      \draw[->] (V2) -- node[label=below:$\rho_2$]{} (F2);
    \end{tikzpicture}
  \end{equation}
  for class $\alpha$ and framing dimension $\vec d = \vec 1 \coloneqq
  (1, 1, 1)$. On $\fN^{\tilde Q(k_1,k_2)}_{\alpha,\vec 1}$, take the
  symmetrized obstruction theory given by symmetrized pullback along
  $\Pi_\alpha$, and the stability condition
  $\tau_{\epsilon,\epsilon,1}$ for $\epsilon > 0$. For sufficiently
  small $\epsilon$ (depending on $\alpha$), there are no strictly
  semistables \cite[Definition 9.4]{Joyce2021} and the semistable
  locus $M_\alpha \coloneqq \fN_{\alpha,\vec 1}^{\tilde
    Q(k_1,k_2),\sst}$ is a quasi-projective scheme. We call it the
  {\it master space}.

  Let $\bC^\times$ act on $M_\alpha$ by scaling the linear map
  $\rho_4$ with a weight which we denote $z$. This action commutes
  with the inherited $\sT$-action and the obstruction theory may be
  made $\bC^\times$-equivariant as well. The quantity
  \[ \chi\left(M_\alpha, \hat\cO^\vir\right) \in \bk_{\bC^\times \times \sT,\loc}[t^{\pm \frac{1}{2}}] \]
  is well-defined by $(\bC^\times \times \sT)$-equivariant
  localization (Lemma~\ref{lem:quiver-stack-fixed-locus-is-proper}).
  We will apply the following
  Proposition~\ref{prop:master-space-relation} to $M_\alpha$ to get a
  wall-crossing formula.
\end{definition}

\subsection{}

\begin{proposition} \label{prop:master-space-relation}
  Let $M$ be a scheme with $(\bC^\times \times \sT)$-action and
  $(\bC^\times \times \sT)$-equivariant symmetric perfect obstruction
  theory for which equivariant localization holds, and let
  $\iota_{\bC^\times}\colon M^{\bC^\times} \hookrightarrow M$ be the
  $\bC^\times$-fixed locus. Suppose that:
  \begin{enumerate}
  \item $\chi(M, \hat\cO^\vir)$ is well-defined and lies in
    $\bk_{\sT,\loc}[t^{\pm \frac{1}{2}}, z^\pm] \subset
    \bk_{\bC^\times \times \sT,\loc}[t^{\pm \frac{1}{2}}]$;
  \item
    $N^\vir_{\iota_{\bC^\times}} = F^\vee - t^{-1} F \in K^\circ_{\bC^\times \times \sT}(M^{\bC^\times})$
    for some K-theory class $F$.
  \end{enumerate}
  Let $F_{>0}$ and $F_{<0}$ be the parts of $F$ with positive and
  negative $\bC^\times$-weight respectively, and set
  $\ind \coloneqq \rank F_{>0} - \rank F_{<0}$. Then
  \begin{equation} \label{eq:master-space-relation}
    0 = \chi\bigg(M^{\bC^\times}, \hat\cO^\vir_{M^{\bC^\times}} \otimes (-1)^{\ind}(t^{\frac{\ind}{2}} - t^{-\frac{\ind}{2}})\bigg).
  \end{equation}
\end{proposition}

The quantity $\ind$ is a kind of Morse--Bott index for connected
components $Z \subset M^{\bC^\times}$. In particular $\ind$ is
constant on each $Z$ and so the scalar
$(-1)^{\ind}(t^{\frac{\ind}{2}} - t^{-\frac{\ind}{2}})$ factors out of
$\chi(Z, \hat\cO^\vir_Z \otimes \cdots)$.

Proposition~\ref{prop:master-space-relation} is a $\sT$-equivariant
and K-theoretic version of a common procedure in cohomology (see e.g.
\cite[\S 5]{Kiem2013}, or \cite[\S 5]{Nakajima2011} for a
$\sT$-equivariant version): if $M$ is proper, take residues of both
sides of the $\bC^\times$-equivariant localization formula for
$[M]^\vir$. Accordingly, the proof of
Proposition~\ref{prop:master-space-relation}, given in
\S\ref{sec:master-space-relation-proof}, is by applying the following
K-theoretic residue map to the $\bC^\times \times \sT$-equivariant
localization formula.

\subsection{}

\begin{definition}
  Given a rational function $f \in \bk_{\bC^\times \times \sT,\loc}$,
  let $f_\pm \in \bk_{\sT,\loc}((z^\pm))$ be its formal series
  expansion around $z=0$ and $z=\infty$ respectively. The {\it
    K-theoretic residue map}, see \cite{Metzler2002} \cite[Appendix
  A]{Liu2022}, is the $\bZ$-module homomorphism
  \begin{align*}
    \Res^K_z\colon \bk_{\bC^\times \times \sT,\loc} &\to \bk_{\sT,\loc} \\
    f &\mapsto z^0 \text{ term in } (f_+ - f_-).
  \end{align*}
  Note that $\Res_z^K\colon \bk_{\sT,\loc}[z^\pm] \mapsto 0$, and
  $\Res_z^K(f) = \lim_{z \to 0} f - \lim_{z \to \infty} f$ whenever
  the right hand side exists.
\end{definition}

\subsection{}
\label{sec:master-space-relation-proof}

\begin{proof}[Proof of Proposition~\ref{prop:master-space-relation}.]
  First, factorize
  \[ \iota\colon M^{\bC^\times \times \sT} \xhookrightarrow{\iota_\sT} M^{\bC^\times} \xhookrightarrow{\iota_{\bC^\times}} M, \]
  and, in K-theory, split
  $N^\vir_\iota = \iota_\sT^* N^\vir_{\iota_{\bC^\times}} + N^\vir_{\iota_\sT}$
  into its non-$\bC^\times$-fixed part and its $\bC^\times$-fixed but
  non-$\sT$-fixed part respectively. Then
  \begin{equation} \label{eq:master-space-localization}
    \chi(M, \hat\cO^\vir_{M}) = \chi\bigg(M^{\bC^\times \times \sT}, \bigg(\frac{\cO^\vir_{M^{\bC^\times \times \sT}}}{\se(N^\vir_{\iota_\sT})} \otimes K^{\vir,\frac{1}{2}}_{M^{\bC^\times}}\bigg) \otimes \frac{1}{\hat\se(\iota_\sT^*N^\vir_{\iota_{\bC^\times}})}\bigg)
  \end{equation}
  using the definition \eqref{eq:symmetrized-virtual-cycle} of
  $\hat\cO^\vir_{M}$ and the discussion of
  \S\ref{sec:symmetrized-virtual-localization} to factor out the
  contribution of $\iota_{\sT}^*N_{\iota_{\bC^\times}}^{\vir}$. Note
  that the inner bracketed term, if pushed forward along $\iota_\sT$,
  becomes $\hat\cO^\vir_{M^{\bC^\times}}$. \footnote{It is therefore
  tempting to rewrite \eqref{eq:master-space-localization} as an
  integral over $M^{\bC^\times}$, by pushing forward the integrand
  along $\iota_\sT$, but then the second term $\hat
  \se(N_{\iota_{\bC^\times}}^\vir) \in K_{\bC^\times \times
    \sT}(M^{\bC^\times})_\loc$ is not invertible and so this doesn't
  make sense.}

  Now apply $\Res^K_z$ to \eqref{eq:master-space-localization}. By
  hypothesis, the left hand side vanishes. On the right hand side, the
  definition \eqref{eq:k-theoretic-euler-class} of K-theoretic Euler
  class gives
  \begin{equation} \label{eq:virtual-normal-bundle-in-localization}
    \frac{1}{\hat \se(\iota^*_\sT N^\vir_{\iota_{\bC^\times}})} = \frac{\hat \se(t^{-1}\otimes \iota_\sT^* F)}{\hat \se(\iota_\sT^* F^\vee)} = \prod_{z^a t^b L \in \iota_\sT^* F} \left(-t^{-\frac{1}{2}}\cdot \frac{t - z^a t^b L}{1 - z^a t^b L}\right).
  \end{equation}
  The first equality follows from the hypothesis that
  $N^\vir_{\iota_{\bC^\times}} = F^\vee - t^{-1} F$, and the second
  equality is the splitting principle. The product ranges over all
  K-theoretic Chern roots of $\iota_\sT^* F \in K^\circ_{\bC^\times
    \times \sT}(M^{\bC^\times \times \sT}) = K^\circ(M^{\bC^\times
    \times \sT}) \otimes \bk_{\bC^\times \times \sT}$. Each such Chern
  root has $a \neq 0$ --- otherwise, it would be $\bC^\times$-fixed ---
  so the $z \to 0$ and $z \to \infty$ limits of each term in the
  product exist and are $-t^{\pm \frac{1}{2}}$ depending on the sign
  of $a$. \footnote{There is a mild subtlety here because the inverse
  of $1 - z^a t^b L$ is a series expansion in $(1 - z^a t^b)^{-1} (1 -
  L)$, which terminates because $1 - L$ is nilpotent, and not a series
  expansion in $z^a t^b L$, which does not terminate because $L$ is
  not necessarily nilpotent. Nonetheless, we can still treat
  quantities like \eqref{eq:virtual-normal-bundle-in-localization}
  formally like rational functions valued in Chern roots.} The result
  is that
  \[ 0 = \chi\bigg(M^{\bC^\times \times \sT}, \bigg(\frac{\cO^\vir_{M^{\bC^\times \times \sT}}}{\se(N^\vir_{\iota_\sT})} \otimes K^{\vir,\frac{1}{2}}_{M^{\bC^\times}}\bigg) \otimes (-1)^{\ind}(t^{\frac{\ind}{2}} - t^{-\frac{\ind}{2}})\bigg). \]
  To conclude, pushforward the integrand along $\iota_\sT$, recognizing
  the resulting bracketed term as the definition of
  $\hat\cO^\vir_{M^{\bC^\times}}$.
\end{proof}

\subsection{}

\begin{lemma} \label{lem:no-z-poles}
  Let $M$ be a scheme with $(\bC^\times \times \sT)$-action and
  $(\bC^\times \times \sT)$-equivariant symmetric perfect obstruction
  theory for which equivariant localization holds. If
  \begin{enumerate}
  \item for the maximal torus $\sT_w \subset \ker(w)$ of any
    $(\bC^\times \times \sT)$-weight $w = z^a t^b$ with $a \neq 0$,
    the fixed locus $M^{\sT_w}$ is proper, and
  \item the obstruction theory is symmetric of weight $t$,
  \end{enumerate}
  then $\chi\big(M, \hat\cO^\vir\big)$ is well-defined and lies in
  $\bk_{\sT,\loc}[t^{\pm \frac{1}{2}}, z^\pm]$.
\end{lemma}

\begin{proof}
  For any $w$ appearing in condition 1, consider the obvious closed
  embeddings
  \[ \iota\colon M^{\bC^\times \times \sT} \hookrightarrow M^{\sT_w} \xrightarrow{\iota_w} M. \]
  Since $M^{\sT_w}$ is proper by hypothesis, so is $M^{\bC^\times
    \times \sT}$, and so
  \[ \chi\big(M, \hat\cO_M^\vir\big) \in \bk_{\bC^\times \times \sT,\loc}[t^{\pm \frac{1}{2}}] \]
  is well-defined. View this as a rational function in $z$ and $t$. By
  the definition \eqref{eq:localized-base-ring} of localized base
  rings, we must show that its denominator has no factors of $1 - w$
  --- equivalently, that the $w = 1$ specialization of $\chi(M,
  \hat\cO^\vir)$ is well-defined.

  We use the following slight modification of a geometric observation
  of \cite[Proposition 3.2]{Arbesfeld2021}. Since $t$ is a
  $\sT$-weight and $\sT \subset \sT_w$, the same argument that shows a
  square root of $\iota^*K_M^\vir$ exists also shows that a square
  root of $\iota_w^*K_M^\vir$ exists. So, by localization with respect
  with a $\bC^\times$ (non-canonical) such that $\sT_w \times
  \bC^\times = \bC^\times \times \sT$,
  \[ \chi(M, \hat\cO_M^\vir) = \chi\left(M^{\sT_w}, \frac{\cO_{M^{\sT_w}}^\vir}{\se(N_{\iota_w}^\vir)} \otimes (\iota_w^*K_M^\vir)^{\frac{1}{2}}\right). \]
  The right hand side is clearly well-defined $\sT_w$-equivariantly,
  i.e. after specializing $w = 1$.
\end{proof}

\subsection{}

\begin{remark}
  The combination of Proposition~\ref{prop:master-space-relation} and
  Lemma~\ref{lem:no-z-poles} works equally well in equivariant
  cohomology, upon replacing the K-theoretic residue map $\Res_z^K$
  with its cohomological analogue $f \mapsto (z^{-1} \text{ term in }
  f_-)$. Working cohomologically, the result is that the term
  $t^{\frac{\ind}{2}} - t^{-\frac{\ind}{2}}$ in
  \eqref{eq:master-space-relation} become $\ind \cdot t$, and the
  resulting quantum integers in the wall-crossing formulas
  \eqref{eq:VW-invariants-relation} and
  \eqref{eq:VW-invariants-relation-simple} become specialized to
  $t=1$.
\end{remark}

\subsection{}

\begin{lemma} \label{lem:master-space-no-z-poles}
  The master space $M_\alpha$ of Definition~\ref{def:master-space}
  satisfies the conditions of Lemma~\ref{lem:no-z-poles}.
\end{lemma}

\begin{proof}
  The hypothesis $a \neq 0$ guarantees that $\sT \subset \sT_w$. Since
  $M_\alpha$ parameterizes triples $(E, \vec V, \vec \rho)$ and $\sT$
  acts on $E$, triples which are $\sT_w$-fixed must in particular have
  $\sT$-fixed $E$. Thus, the same argument as in
  Lemma~\ref{lem:quiver-stack-fixed-locus-is-proper} shows that
  $M_\alpha^{\sT_w}$ is proper, so condition 1 is satisfied.

  Symmetrized pullback of a symmetric obstruction theory of weight $t$
  is still symmetric of weight $t$, see e.g.
  \eqref{eq:symmetrized-pullback-k-class}, so condition 2 is also
  satisfied.
\end{proof}

\subsection{}

By Lemma~\ref{lem:master-space-no-z-poles}, we may apply
Proposition~\ref{prop:master-space-relation} and
Lemma~\ref{lem:no-z-poles} to $M_\alpha$. In
\S\ref{sec:master-space-fixed-loci-1},
\S\ref{sec:master-space-fixed-loci-2}, and
\S\ref{sec:master-space-fixed-loci-3}, we identify
$M_\alpha^{\bC^\times}$ as a disjoint union of three types of loci;
see \cite[Proposition 9.5]{Joyce2021} for details. All three types of
loci are related in very manageable ways to pairs stacks
$\fN^{Q(k)}_{\alpha,1}$, and \eqref{eq:master-space-relation} will
become a wall-crossing formula relating the refined pairs invariants
$\tilde\VW_\alpha(k_1, t)$ and $\tilde\VW_\alpha(k_2, t)$.

On each fixed locus, we also need to identify the $\bC^\times$-fixed
and non-$\bC^\times$-fixed parts of the restriction of the obstruction
theory on $M_\alpha$. In K-theory, the obstruction theory is given by
\begin{equation} \label{eq:master-space-obstruction-theory}
  \begin{aligned}
    \bT_{M_\alpha}
    &= \Big(\left(\cV_3^\vee \otimes \cV_1 + \cV_3^\vee \otimes \cV_2 - \cV_3^\vee \otimes \cV_3\right) - t^{-1} \otimes (\cdots)^\vee\Big) \\
    &\quad+ \sum_{i=1}^2 \Big(\left(\cV_i^\vee \otimes \cF_{k_i}(\cE) - \cV_i^\vee \otimes \cV_i\right) - t^{-1} \otimes (\cdots)^\vee\Big) - R\cHom_{\pi_X}(\cE, \cE)_\perp
  \end{aligned}
\end{equation}
where $\cE$ is the universal sheaf of $E$ on
$\pi_X\colon \fN \times X \to \fN$, and $\cF_{k_i}(\cE)$ is the
universal bundle of $F_{k_i}(E)$, and each $(\cdots)^\vee$ indicates
the dual of the preceding term. This formula can be obtained by
combining \eqref{eq:symmetrized-pullback-k-class} and
\eqref{eq:quiver-stack-projection}.

For clarity in keeping track of $\bC^\times$-weights on
$\bC^\times$-fixed loci, for components
$Z \subset M_\alpha^{\bC^\times}$, let $\cW_i$ denote the
non-$\bC^\times$-equivariant bundle on $Z$ such that
$\cV_i\big|_Z = z^a \otimes \cW_i$ for some $a \in \bZ$.

\subsection{}
\label{sec:master-space-fixed-loci-1}

Let $Z_{\rho_4=0} \coloneqq \{\rho_4=0\} \subset M_\alpha$. This locus
is obviously $\bC^\times$-fixed. By stability, $\rho_3 \neq 0$ and the
forgetful map
\[ \Pi_{\rho_4=0}\colon Z_{\rho_4=0} \to \fN^{Q(k_1),\sst}_{\alpha,1}, \quad (E, \vec V, \vec \rho) \mapsto (E, V_1, \rho_1) \]
is a $\bP^{\lambda_{k_2}(\alpha)-1}$-bundle, coming from the freedom
to choose the map $\rho_2$.

Now consider the restriction of $\bT_{M_\alpha}$. The non-vanishing
section $\rho_3$ trivializes the line bundle
$\cV_3^\vee \otimes \cV_1$, which then cancels with
$\cV_3^\vee \otimes \cV_3 \cong \cO$ in the first line of
\eqref{eq:master-space-obstruction-theory}. What remains is the
non-$\bC^\times$-fixed part
\begin{equation} \label{eq:master-space-fixed-locus-1-Nvir}
  N^{\vir}_{\iota_{\bC^\times}}\Big|_{Z_{\rho_4=0}} = z \cW_3^\vee \otimes \cW_2 - t^{-1} \otimes (\cdots)^\vee.
\end{equation}
Finally, the entire second line of
\eqref{eq:master-space-obstruction-theory} is $\bC^\times$-fixed. So,
the computation in the proof of
Proposition~\ref{prop:symmetrized-integration-on-projective-bundle}
applies to $\Pi_{\rho_4=0}$, yielding
\begin{equation} \label{eq:master-space-fixed-locus-1}
  \chi\left(Z_{\rho_4=0}, \hat\cO^\vir\right) = [\lambda_{k_2}(\alpha)]_t \cdot \chi\left(\fN_{\alpha,1}^{Q(k_1),\sst}, \hat\cO^\vir\right) = [\lambda_{k_2}(\alpha)]_t \cdot \tilde \VW_\alpha(k_1, t).
\end{equation}

\subsection{}
\label{sec:master-space-fixed-loci-2}

Let $Z_{\rho_3=0} \coloneqq \{\rho_3=0\} \subset M_\alpha$. This locus
is $\bC^\times$-fixed because the scaling of $\rho_4$ can be undone by
making $\bC^\times$ act on $V_3$ with weight $z$. By stability,
$\rho_4 \neq 0$ and the forgetful map
\[ \Pi_{\rho_3=0}\colon Z_{\rho_3=0} \to \fN^{Q(k_2),\sst}_{\alpha,1}, \quad (E, \vec V, \vec \rho) \mapsto (E, V_2, \rho_2) \]
is a $\bP^{\lambda_{k_1}(\alpha)-1}$-bundle, coming from the freedom
to choose the map $\rho_1$.

Now consider the restriction of $\bT_{M_\alpha}$. Like in
\S\ref{sec:master-space-fixed-loci-1} for $Z_{\rho_4=0}$, the
non-$\bC^\times$-fixed part is
\begin{equation} \label{eq:master-space-fixed-locus-2-Nvir}
  N^\vir_{\iota_{\bC^\times}}\Big|_{Z_{\rho_3=0}} = z^{-1} \cW_3^\vee \otimes \cW_1 - t^{-1} \otimes (\cdots)^\vee,
\end{equation}
where the $\bC^\times$ action on $V_3$ creates the non-trivial $z$
weight. The entire second line of
\eqref{eq:master-space-obstruction-theory} is $\bC^\times$-fixed, and
\begin{equation} \label{eq:master-space-fixed-locus-2}
  \chi\left(Z_{\rho_3=0}, \hat\cO^\vir\right) = [\lambda_{k_1}(\alpha)]_t \cdot \chi\left(\fN_{\alpha,1}^{Q(k_2),\sst}, \hat\cO^\vir\right) = [\lambda_{k_1}(\alpha)]_t \cdot \tilde \VW_\alpha(k_2, t).
\end{equation}

\subsection{}
\label{sec:master-space-fixed-loci-3}

Finally, when both $\rho_3,\rho_4 \neq 0$, by stability all
$\rho_i \neq 0$. In order for all maps to be $\bC^\times$-equivariant,
$\bC^\times$ must act on $V_1$, $V_2$, and $V_3$ with weights $z$,
$1$, and $z$ respectively, and $E = zE_1 \oplus E_2$ must split into
weight-$z$ and weight-$1$ pieces \footnote{No extra summands can exist
  in the $\bC^\times$-weight decomposition of $E$, e.g. a weight-$z^2$
  piece, because they would be destabilizing.} such that $\rho_i$
factors as $\rho_i\colon V_i \to F_{k_i}(E_i)$ for $i = 1, 2$.
Semistability of $E$ implies $\tau(E_1) = \tau(E_2)$. Note that while
$\det \bar E = \det(\bar E_1) \det(\bar E_2) = L \in \Pic(S)$ is
pre-specified, $\det(\bar E_1)$ and $\det(\bar E_2)$ themselves are
not. Hence there are fixed loci
\[ Z_{\alpha_1,\alpha_2} \coloneqq \{\det(\bar E_1)\det(\bar E_2) = L\} \subset \fM^{Q(k_1),\sst}_{\alpha_1,1} \times \fM^{Q(k_2),\sst}_{\alpha_2,1} \subset M_\alpha \]
for any non-trivial decomposition $\alpha = \alpha_1 + \alpha_2$ with
$\tau(\alpha_1) = \tau(\alpha_2)$. Put differently, if
$\alpha_i = (r_i, c_{1,i}, c_{2,i})$, fix a splitting
$L_1 \otimes L_2 = L$ where $c_1(L_i) = c_{1,i}$, and then there is a
map $\det_1\colon Z_{\alpha_1,\alpha_2} \to \Pic_0(S)$, given by
$L_1^\vee \otimes \det \bar E_1$, whose fiber at $L_0$ is
\[ \det\nolimits_1^{-1}(L_0) \cong \fN^{Q(k_1),\sst}_{(r_1,L_1 \otimes L_0,c_{2,1}),1} \times \fN^{Q(k_2),\sst}_{(r_2, L_2 \otimes L_0^{-1}, c_{2,2}),1}. \]

Now consider the restriction of $\bT_{M_\alpha}$. The non-vanishing
sections $\rho_3$ and $\rho_4$ trivialize the line bundles
$\cV_3^\vee \otimes \cV_1$ and $\cV_3^\vee \otimes \cV_2$, so the
first line of \eqref{eq:master-space-obstruction-theory} becomes
$\cO - t\cO^\vee$. The non-trivial $\bC^\times$-weight in the
splitting
$\cE\big|_{Z_{\alpha_1,\alpha_2}} = z\cE_1 \oplus \cE_2$
produces the non-$\bC^\times$-fixed part
\begin{equation} \label{eq:master-space-fixed-locus-3-Nvir}
  \begin{aligned}
    N^{\vir}_{\iota_{\bC^\times}}\Big|_{Z_{\alpha_1,\alpha_2}}
    &= \left(z^{-1} \cW_1^\vee \otimes \cF_{k_1}(\cE_2) + z \cW_2^\vee \otimes \cF_{k_2}(\cE_1)\right) - t \otimes (\cdots)^\vee \\
    &\qquad- \left(z^{-1} R\cHom_{\pi_X}(\cE_1, \cE_2) + z R\cHom_{\pi_X}(\cE_2, \cE_1)\right)
  \end{aligned}
\end{equation}
where the subscript $\perp$ can be removed because trace does not see
off-diagonal parts. Recall that Serre duality says
$\chi(\alpha_1, \alpha_2) \coloneqq \rank R\cHom_{\pi_X}(\cE_1, \cE_2)$
is anti-symmetric. The remaining $\bC^\times$-fixed part of the second
line of \eqref{eq:master-space-obstruction-theory}, combined with the
$\cO - t\cO^\vee$ from the first line, gives the total
$\bC^\times$-fixed part
\begin{align}
  &\sum_{i=1}^2 \Big(\left(\cW_i^\vee \otimes \cF_{k_i}(\cE_i) - \cW_i^\vee \otimes \cW_i^\vee\right) - t \otimes (\cdots)^\vee\Big) \nonumber \\
  &- \Big(R\cHom_{\pi_X}(\cE_1, \cE_1) + R\cHom_{\pi_X}(\cE_2, \cE_2)\Big)_\perp + \cO - t\cO^\vee. \label{eq:split-fixed-loci-fixed-Ext-term}
\end{align}
Non-canonically, the second line is
$-R\cHom_{\pi_X}(\cE_1, \cE_1)_\perp - (R\cHom_{\pi_X}(\cE_2, \cE_2) - \cO + t\cO^\vee)$.

\subsection{}
\label{sec:mixed-fixed-locus-cases}

There are now three cases for the computation of
$\chi(Z_{\alpha_1,\alpha_2}, \hat\cO^\vir)$.

If $H^1(\cO_S) = H^2(\cO_S) = 0$, then from
\eqref{eq:VW-reduced-obstruction-theory},
$R\cHom_{\pi_X}(\cE_2, \cE_2)_\perp = R\cHom_{\pi_X}(\cE_2, \cE_2) - \cO + t\cO^\vee$,
and so the $\bC^\times$-fixed part
\eqref{eq:split-fixed-loci-fixed-Ext-term} becomes, canonically,
\[ -R\cHom_{\pi_X}(\cE_1, \cE_1)_\perp - R\cHom_{\pi_X}(\cE_2, \cE_2)_\perp. \]
The base $\Pic_0(S)$ of the map $\det_1$ is trivial. Hence
\begin{equation} \label{eq:master-space-fixed-locus-3}
  \chi\left(Z_{\alpha_1,\alpha_2}, \hat\cO^\vir\right) = \chi\left(\fN^{Q(k_1),\sst}_{\alpha_1,1} \times \fN^{Q(k_2),\sst}_{\alpha_2,1}, \hat\cO^\vir \boxtimes \hat\cO^\vir\right) = \tilde \VW_{\alpha_1}(k_1, t) \tilde \VW_{\alpha_2}(k_2, t).
\end{equation}

If $H^2(\cO_S) \neq 0$, then the obstruction sheaf in
\eqref{eq:split-fixed-loci-fixed-Ext-term} has the extra trivial
summands $R^2\pi_{S*}\cO$ from
$R\cHom_{\pi_S}(\bar \cE, \bar \cE) = R\cHom_{\pi_S}(\bar \cE, \bar \cE)_0 \oplus R^2\pi_{S*}\cO \oplus \cdots$,
where $\bar \cE$ is the universal sheaf of $\bar E$ on
$\pi_S\colon\fN \times S \to \fN$. Hence
$\hat\cO^\vir_{Z_{\alpha_1,\alpha_2}} = 0$.

Finally, if $H^1(\cO_S) \neq 0$, then the base $\Pic_0(S)$ is
non-trivial. Following the vanishing argument in
\S\ref{sec:U-pairs-invariants}, $\Pic_0(S)$ acts with no fixed points
on the moduli stack
\[ \{\det(\bar E_1) \det(\bar E_2) = L\} \subset \fM_{\alpha_1} \times \fM_{\alpha_2}, \]
giving a $\dim H^1(\cO_S)$-dimensional space of nowhere-vanishing
``$t$-twisted cosections'' of the obstruction sheaf in
\eqref{eq:split-fixed-loci-fixed-Ext-term}.
Lemma~\ref{lem:symmetrized-pullback-cosection} lifts this to
nowhere-vanishing ``$t$-twisted cosections'' of the obstruction sheaf
on $Z_{\alpha_1,\alpha_2}$, so $\hat\cO^\vir_{Z_{\alpha_1,\alpha_2}}
\equiv 0 \bmod{(1 - t)^{\dim H^1(\cO_S)}}$ by cosection localization
\cite{Kiem2013a}.

\subsection{}

We are ready to put everything together. Plugging
\eqref{eq:master-space-fixed-locus-1},
\eqref{eq:master-space-fixed-locus-2}, and
\eqref{eq:master-space-fixed-locus-3} into
\eqref{eq:master-space-relation}, if $H^1(\cO_S) = H^2(\cO_S) = 0$
then, after dividing by an overall factor of
$t^{\frac{1}{2}} - t^{-\frac{1}{2}}$,
\begin{equation} \label{eq:VW-invariants-relation}
  \begin{aligned}
    0 &= [\lambda_{k_2}(\alpha)]_t \tilde \VW_\alpha(k_1,t) - [\lambda_{k_1}(\alpha)]_t \tilde \VW_\alpha(k_2, t) \\
    &\quad + \sum_{\substack{\alpha_1+\alpha_2=\alpha\\\forall j: \, \tau(\alpha_j)=\tau(\alpha)}} [\lambda_{k_2}(\alpha_1) - \lambda_{k_1}(\alpha_2) + \chi(\alpha_1, \alpha_2)]_t \tilde \VW_{\alpha_1}(k_1,t) \tilde\VW_{\alpha_2}(k_2,t).
  \end{aligned}
\end{equation}
Note that $\ind = -1$ on $Z_{\rho_4=0}$ by
\eqref{eq:master-space-fixed-locus-1-Nvir}, $\ind = +1$ on
$Z_{\rho_3=0}$ by \eqref{eq:master-space-fixed-locus-2-Nvir}, and
$\ind = \lambda_{k_2}(\alpha_1) - \lambda_{k_1}(\alpha_2) +
\chi(\alpha_1,\alpha_2)$ on $Z_{\alpha_1,\alpha_2}$ by
\eqref{eq:master-space-fixed-locus-3-Nvir}. Also, recall that
$\tau(\alpha_1) = \tau(\alpha_2)$ if and only if $\tau(\alpha_1) =
\tau(\alpha) = \tau(\alpha_2)$; this is a general property of
stability conditions.

Otherwise, if $H^2(\cO_S) \neq 0$, by the second case in
\S\ref{sec:mixed-fixed-locus-cases}, there is no contribution from the
sum and only the first two terms remain:
\begin{equation} \label{eq:VW-invariants-relation-simple}
  0 = [\lambda_{k_2}(\alpha)]_t \tilde \VW_\alpha(k_1,t) - [\lambda_{k_1}(\alpha)]_t \tilde \VW_\alpha(k_2, t).
\end{equation}
Finally, if $H^1(\cO_S) \neq 0$, by the third case in
\S\ref{sec:mixed-fixed-locus-cases}, the same is true modulo $(1 -
t)^{\dim H^1(\cO_S)}$.

If one further assumes that $\cO_S(1)$ is generic for $\alpha$ in the
sense of \eqref{eq:generic-polarization}, then in the sum in
\eqref{eq:VW-invariants-relation}, $\alpha_2 = c \alpha_1$ for some
constant $c$, and therefore
\[ \chi(\alpha_1, \alpha_2) = c \chi(\alpha_1, \alpha_1) = 0 \]
by anti-symmetry of $\chi$, so $\chi(\alpha_1, \alpha_2)$ does not
contribute in the quantum integer. This assumption can be used to
simplify the combinatorics in \S\ref{sec:semistable-invariants} but is
ultimately unnecessary.

\section{Defining semistable invariants}
\label{sec:semistable-invariants}

\subsection{}

Using \eqref{eq:VW-invariants-relation-simple}, the proof of cases 2
and 3 of the main Theorem~\ref{thm:VW-invars} is clear: since
$\lambda_k(\alpha) = \chi(\alpha(k))$ is positive,
$[\lambda_k(\alpha)]_t$ is invertible and $\tilde\VW_\alpha(k,
t)/[\lambda_k(\alpha)]_t$ is independent of $k$. It remains to use the
more complicated relation \eqref{eq:VW-invariants-relation} to prove
case 1. This is a combinatorial result which we prove in this section
as Corollary~\ref{cor:VW-semistable-invariants}. (See also
Remark~\ref{rem:artificial-lie-algebra}, which motivates
Definition~\ref{def:artificial-lie-algebra}.)

\subsection{}

\begin{definition} \label{def:artificial-lie-algebra}
  Let $(A, +)$ be a commutative monoid, and let
  \[ \scA \coloneqq \bigoplus_{(k,\alpha) \in \bZ \times A} \scA_{k,\alpha} \coloneqq \bQ(t^{\frac{1}{2}}) \left[\{\tilde \sZ_{k,0}^\pm\}_k \cup \{\tilde \sZ_{k,\alpha}\}_{k,\alpha}\right] \]
  be the $(\bZ \times A)$-graded $\bQ(t^{\frac{1}{2}})$-algebra where
  $\tilde \sZ_{k,\alpha}$ has degree $(k, \alpha) \in \bZ \times A$,
  and $\scA_{k,\alpha} \subset \scA$ is the $(k,\alpha)$-weight part.
  Given a set map $\tilde\chi\colon (\bZ \times A)^2 \to \bZ$, define
  \begin{align*}
    [-, -]\colon \scA_{k_1,\alpha_1} \times \scA_{k_2,\alpha_2} &\to \scA_{k_1+k_2,\alpha_1+\alpha_2} \\
    (x, y) &\mapsto \left[\tilde\chi\left((k_1,\alpha_1),(k_2,\alpha_2)\right)\right]_t \cdot xy
  \end{align*}
  and extend it to $[-, -]\colon \scA \times \scA \to \scA$
  bilinearly.
\end{definition}

\subsection{}

\begin{lemma}
  If $\tilde\chi$ is bilinear and anti-symmetric, then $[-, -]$ is a
  Lie bracket.
\end{lemma}

\begin{proof}
  Anti-symmetry of $\tilde\chi$ makes $[-, -]$ anti-symmetric. It remains to
  verify the Jacobi identity
  \[ [x_1, [x_2, x_3]] + [x_2, [x_3, x_1]] + [x_3, [x_1, x_2]] = 0 \]
  for $x_i \in \scA_{k_i,\alpha_i}$. Set
  $\tilde\chi_{ij} \coloneqq \tilde\chi((k_i, \alpha_i), (k_j,\alpha_j))$,
  and let
  $(a,b,c) \coloneqq (\tilde\chi_{12}, \tilde\chi_{13}, \tilde\chi_{23})$
  for short. Then we must prove
  \begin{equation} \label{eq:q-integer-jacobi-identity}
    [\tilde\chi_{12} + \tilde\chi_{13}]_t [\tilde\chi_{23}]_t - [\tilde\chi_{23} - \tilde\chi_{12}]_t [\tilde\chi_{13}]_t - [\tilde\chi_{13} + \tilde\chi_{23}]_t [\tilde\chi_{12}]_t = 0.
  \end{equation}
  This is true for arbitrary integers $a, b, c \in \bZ$, not just $(a,
  b, c) = (\tilde\chi_{12}, \tilde\chi_{13}, \tilde\chi_{23})$, using
  the quantum integer addition formula $[a+b]_t = (-1)^a t^{\pm
    \frac{a}{2}} [b]_t + (-1)^b t^{\mp \frac{b}{2}} [a]_t$ to expand
  all terms. This addition formula arises from $t^{\pm \frac{a+b}{2}}
  - t^{\mp \frac{a+b}{2}} = t^{\pm \frac{a}{2}} (t^{\pm \frac{b}{2}} -
  t^{\mp \frac{b}{2}}) + t^{\mp \frac{b}{2}} (t^{\pm \frac{a}{2}} -
  t^{\mp \frac{a}{2}})$ by multiplying both sides by $\pm
  (-1)^{a+b-1}/(t^{\frac{1}{2}} - t^{-\frac{1}{2}})$.
\end{proof}

\subsection{}

\begin{proposition} \label{prop:semistable-invariants}
  Fix $k_1, k_2 \in \bZ$. Suppose $\tilde\chi$ is bilinear and
  anti-symmetric, and
  \[ \tilde\chi\left((k_1, 0), (k_2, 0)\right) = 0, \quad \tilde\chi\left((0,\alpha),(k_i,0)\right) \neq 0 \]
  for all $0 \neq \alpha \in A$ and $i = 1, 2$. Moreover, suppose
  there exists a homomorphism $\rank\colon A \to \bZ_{\ge 0}$ such
  that $\rank \alpha > 0$ for $0 \neq \alpha \in A$. Then, in the
  quotient of $\scA$ by
  \begin{equation} \label{eq:invariants-relation}
    \left[\tilde \sZ_{k_1,\alpha}, \tilde \sZ_{k_2,0}\right] + \left[\tilde \sZ_{k_1,0}, \tilde \sZ_{k_2,\alpha}\right] + \sum_{\substack{\alpha_1+\alpha_2=\alpha\\\forall j: \, \tau(\alpha_j)=\tau(\alpha)}} \left[\tilde \sZ_{k_1,\alpha_1}, \tilde \sZ_{k_2,\alpha_2}\right] = 0, \qquad \forall 0 \neq \alpha \in A,
  \end{equation}
  where $\tau$ is any function on $A$, the elements
  $\{\sZ^{(i)}_\alpha\}_{\alpha \neq 0} \in \scA_{0,\alpha}$ uniquely
  defined by the formulas
  \begin{equation} \label{eq:semistable-invariants}
    \tilde \sZ_{k_i,\alpha} = \sum_{\substack{n>0\\\alpha_1 + \cdots + \alpha_n = \alpha\\\forall j: \tau(\alpha_j)=\tau(\alpha)}} \frac{1}{n!} \left[\sZ^{(i)}_{\alpha_n}, \left[\cdots \left[ \sZ^{(i)}_{\alpha_2}, [\sZ^{(i)}_{\alpha_1}, \tilde \sZ_{k_i,0}]\right] \cdots \right]\right]
  \end{equation}
  are independent of $i$, namely $\sZ^{(1)}_\alpha = \sZ^{(2)}_\alpha$
  for all $0 \neq \alpha \in A$.
\end{proposition}

For instance, if $\rank \alpha = 1$, then $\alpha$ is indecomposable
and the right hand side of \eqref{eq:semistable-invariants} consists
of only the $n=1$ term $[\sZ^{(i)}_\alpha, \tilde \sZ_{k_i,0}]$. Since
$\tilde \sZ_{k_i,0}$ is invertible, $\sZ^{(i)}_\alpha \propto
\tilde\sZ_{k_i,\alpha} \tilde \sZ^{-1}_{k_i,0} \in \scA_{0,\alpha}$,
and then one can determine the scalar by the definition of the Lie
bracket:
\[ \tilde \sZ_{k_i,\alpha} = [\sZ^{(i)}_\alpha, \tilde \sZ_{k_i,0}] = [\tilde\chi((0,\alpha), (k_i,0))]_t \cdot \sZ^{(i)}_\alpha \tilde \sZ_{k_i,0}. \]
This uniquely defines $\sZ^{(i)}_\alpha$, and the relation
\eqref{eq:invariants-relation} immediately implies
$\sZ^{(1)}_\alpha = \sZ^{(2)}_\alpha$.

\subsection{}

\begin{corollary} \label{cor:VW-semistable-invariants}
  Given elements
  $\{\tilde \VW_\alpha(k_1, t), \tilde \VW_\alpha(k_2, t)\}_\alpha \subset \bQ(t^{\frac{1}{2}})$
  satisfying \eqref{eq:VW-invariants-relation},
  \begin{equation} \label{eq:VW-semistable-invariants}
    \tilde \VW_{\alpha}(k_i, t) \eqqcolon \sum_{\substack{n>0\\\alpha_1 + \cdots + \alpha_n = \alpha\\\forall j: \tau(\alpha_j) = \tau(\alpha)}} \frac{1}{n!} \prod_{j=1}^n \Big[\lambda_{k_i}(\alpha_j) - \chi\Big(\sum_{k=1}^{j-1} \alpha_k, \alpha_j\Big)\Big]_t \VW_{\alpha_j}(k_i, t)
  \end{equation}
  uniquely defines elements
  $\{\VW_\alpha(k_i, t)\}_{\alpha} \subset \bQ(t^{\frac{1}{2}})$
  which are independent of $i = 1, 2$.
\end{corollary}

\begin{proof}
  Take $A$ to be the monoid of effective classes \eqref{eq:classes}.
  For the Lie bracket, take the quantity
  \[ \tilde\chi\left((k_1,\alpha_1), (k_2,\alpha_2)\right) \coloneqq \lambda_{k_2}(\alpha_1) - \lambda_{k_1}(\alpha_2) + \chi(\alpha_1,\alpha_2). \]
  which appears in \eqref{eq:VW-invariants-relation}. The function
  $\tau$ is the Gieseker stability function, and the rank homomorphism
  is just the rank, i.e. the $H^0(S)$ component, of the class
  $\alpha$. The specialization
  \[ \tilde \sZ_{k,\alpha} \mapsto \begin{cases} \tilde \VW_\alpha(k, t) & \alpha \neq 0 \\ 1 & \alpha = 0 \end{cases} \]
  is valid because the required relation
  \eqref{eq:invariants-relation} becomes the wall-crossing formula
  \eqref{eq:VW-invariants-relation}. Then
  \eqref{eq:semistable-invariants} becomes
  \eqref{eq:VW-semistable-invariants}, as desired.
\end{proof}

\subsection{}

\begin{proof}[Proof of Proposition~\ref{prop:semistable-invariants}.]
  We will induct on $\rank \alpha$ to show that $\sZ_\alpha^{(i)} \in
  \scA_{0,\alpha}$ is uniquely defined, and that $\sZ_\alpha^{(1)} =
  \sZ_\alpha^{(2)}$. By the induction hypothesis, all the
  $\sZ_{\alpha_j}^{(i)}$ in the $n>1$ terms of
  \eqref{eq:semistable-invariants} may be replaced by
  $\sZ_{\alpha_j}^{(1)}$:
  \begin{equation} \label{eq:semistable-invariants-inductive}
    \tilde \sZ_{k_i,\alpha} = [\sZ_{\alpha}^{(i)}, \tilde \sZ_{k_i,0}] + \sum_{\substack{n>1\\\alpha_1 + \cdots + \alpha_n = \alpha\\\forall j: \tau(\alpha_j)=\tau(\alpha)}} \frac{1}{n!} \left[\sZ^{(1)}_{\alpha_n}, \left[\cdots \left[ \sZ^{(1)}_{\alpha_2}, [\sZ^{(1)}_{\alpha_1}, \tilde \sZ_{k_i,0}]\right] \cdots \right]\right],
  \end{equation}
  and the sum lives in $\scA_{k_i,\alpha}$. Then $[\sZ^{(i)}_\alpha,
    \tilde \sZ_{k_i,0}] \propto \sZ_\alpha^{(i)} \tilde \sZ_{k_i,0}
  \in \scA_{k_i,\alpha}$ as well. Since $\tilde \sZ_{k_i,0}$ is
  invertible, this uniquely defines $\sZ^{(i)}_\alpha \in
  \scA_{0,\alpha}$. To show $\sZ_\alpha^{(1)} = \sZ_\alpha^{(2)}$,
  plug \eqref{eq:semistable-invariants-inductive} into
  \eqref{eq:invariants-relation} and use the hypothesis on
  $\tilde\chi$ to get
  \begin{equation} \label{eq:semistable-invariants-comparison}
    \left[\tilde\chi\left((0,\alpha),(k_1,0)\right)\right]_t \left[\tilde\chi\left((0,\alpha),(k_2,0)\right)\right]_t \left(\sZ_\alpha^{(1)} - \sZ_\alpha^{(2)}\right)\tilde\sZ_{k_1,0} \tilde\sZ_{k_2,0} + \sum_{\substack{n>1\\\alpha_1+\cdots+\alpha_n=\alpha\\\forall j:\tau(\alpha_j)=\tau(\alpha)}} \!\!C_{\alpha_1,\ldots,\alpha_n} = 0
  \end{equation}
  where the first two terms are the $n=1$ case of the sum
  \eqref{eq:semistable-invariants-inductive} and, with the
  abbreviations $z_\beta \coloneqq \sZ_\beta^{(1)}$ and
  $x \coloneqq \tilde\sZ_{k_1,0}$ and
  $y \coloneqq \tilde\sZ_{k_2,0}$,
  \[ C_{\alpha_1,\ldots,\alpha_n} \coloneqq \sum_{m=0}^n \bigg[\frac{1}{m!} \Big[z_{\alpha_m}, \big[\cdots[z_{\alpha_2}, [z_{\alpha_1}, x]]\cdots\big]\Big], \frac{1}{(n-m)!} \Big[z_{\alpha_n}, \big[\cdots[z_{\alpha_{m+2}}, [z_{\alpha_{m+1}}, y]]\cdots\big]\Big]\bigg]. \]
  To complete the inductive step, it remains to show that the sum in
  \eqref{eq:semistable-invariants-comparison} vanishes. Namely, since
  the coefficient of $\sZ_\alpha^{(1)} - \sZ_\alpha^{(2)}$ is
  invertible by hypothesis, we can then conclude that
  $\sZ_\alpha^{(1)} = \sZ_\alpha^{(2)}$.

  Let $U(\scA) = \bigoplus_{k,\alpha} U(\scA)_{k,\alpha}$ be the
  universal enveloping algebra of $\scA$. It inherits the
  $(\bZ \times A)$-grading since the Lie bracket is homogeneous. On
  $\hat U(\scA) \coloneqq \prod_{k,\alpha} U(\scA)_{k,\alpha}$, i.e.
  the completion of $U(\scA)$ with respect to the grading, consider
  the operator $\ad_z(-) \coloneqq [z, -]$ where
  $z \coloneqq \sum_{\beta: \tau(\beta)=\tau(\alpha)} z_\beta$. Then
  \[ \sum_{\substack{n\ge 0\\\alpha_1+\cdots+\alpha_n=\alpha\\\forall j: \tau(\alpha_j)=\tau(\alpha)}} C_{\alpha_1,\ldots,\alpha_n} = (k_1+k_2,\alpha)\text{-weight piece of } \left[e^{\ad_z} x, e^{\ad_z} y\right] \in \hat U(\scA). \]
  A standard combinatorial result in Lie theory says that $e^{\ad_u} v
  = e^u v e^{-u}$ for any $u, v \in \hat U(\scA)$ \cite[\S 3.9,
    Exercise 14]{Hall2015}, so
  \begin{equation} \label{eq:semistable-invariant-combinatorics}
    \left[e^{\ad_z} x, e^{\ad_z} y\right] = [e^z x e^{-z}, e^z y e^{-z}] = e^z [x, y] e^{-z} = 0.
  \end{equation}
  The hypothesis on $\tilde\chi$ implies $[x, y] = 0$, whence the last
  equality follows. We are done since $C_\emptyset = [x, y] = 0$ by
  definition, and $C_{\alpha_1} = 0$ by the Jacobi identity.
\end{proof}

\subsection{}

\begin{remark} \label{rem:artificial-lie-algebra}
  The reader may wonder whether it was truly necessary to introduce
  the Lie algebra $(\scA, [-, -])$, and to prove
  Corollary~\ref{cor:VW-semistable-invariants} via the more abstract
  Proposition~\ref{prop:semistable-invariants}. The answer is that a
  surprising amount of combinatorics is hidden in
  \eqref{eq:semistable-invariant-combinatorics} and the quantum
  integer identity \eqref{eq:q-integer-jacobi-identity} which made
  $[-, -]$ into a Lie bracket. One could directly substitute
  \eqref{eq:VW-semistable-invariants} into
  \eqref{eq:VW-invariants-relation} and the first half of the proof of
  Proposition~\ref{prop:semistable-invariants} would work equally
  well, but then one must prove the vanishing of the sum of the
  combinatorial coefficients
  \begin{equation} \label{eq:numerical-combinatorics}
    \begin{aligned}
      C_{\alpha_1,\ldots,\alpha_n}
      &= \sum_{m=0}^n \binom{n}{m} \Big[\sum_{i=1}^m b_i - \sum_{j=m+1}^n a_j + \sum_{i=1}^m \sum_{j=m+1}^n c_{ij} \Big]_t \\
      &\qquad\qquad\qquad\prod_{i=1}^m [a_i - \sum_{k=1}^{i-1} c_{k,i}]_t \prod_{j=m+1}^n [b_j - \sum_{k=m+1}^{j-1} c_{k,j}]_t,
    \end{aligned}
  \end{equation}
  where $a_i \coloneqq \lambda_{k_1}(\alpha_i)$ and
  $b_j \coloneqq \lambda_{k_2}(\alpha_j)$ and
  $c_{ij} \coloneqq \chi(\alpha_i, \alpha_j)$. In fact, it is true for
  arbitrary integers $a_i$, $b_j$, and $c_{ij} = -c_{ji}$ that
  \[ \sum_{\sigma \in S_n} C_{\alpha_{\sigma(1)}, \ldots, \alpha_{\sigma(n)}} = 0, \]
  but it appears fairly messy and complicated to use quantum integer
  identities to rearrange the terms and identify which ones cancel
  with which. \footnote{A manageable case is when all the $c_{ij}$ are
    zero, which occurs if $\cO_S(1)$ is generic for $\alpha$. Then the
    terms in \eqref{eq:numerical-combinatorics} for
    $C_{\alpha_1,\ldots,\alpha_n}$ are invariant under
    $S_m \times S_{n-m}$, and so the sum may be rewritten as
    $\sum_{I \sqcup J = \{1, \ldots, n\}}$, and a single application
    of the quantum integer addition formula suffices. Details are left
    to the interested reader.} The Lie bracket formalism helps to
  disentangle the quantum integer identities from the rearrangement
  cancellations.
\end{remark}

\phantomsection
\addcontentsline{toc}{section}{References}

\begin{small}
\bibliographystyle{alpha}
\bibliography{VW-invars}
\end{small}

\end{document}